\documentclass[12pt,reqno]{amsart}
\usepackage{amssymb}
\usepackage[all]{xy}
\usepackage{enumerate}

\newcommand\Mf{{\mathfrak M}}

\newcommand{\C}{{\mathbb C}}

\newcommand{\Z}{{\mathbb Z}}

\newcommand\F{{\mathcal F}}

\newcommand\Ker{\operatorname{Ker}}

\newcommand\End{\operatorname{End}}

\newcommand\Vir{\operatorname{\bf Vir}}
\newcommand\Witt{{\operatorname{\bf Witt}}}

\newcommand\Hom{\operatorname{Hom}}
\newcommand\homparlie{\operatorname{Hom}_{\text{\bf ParLie-alg}}}
\newcommand\Par{\operatorname{Par}}
\newcommand\Pa{\operatorname{P}}

\newcommand\ad{\operatorname{ad}}
\newcommand\supp{\operatorname{supp}}

\renewcommand{\sl}{\mathfrak{sl}}

\newcommand{\g}{\mathfrak{g}}

\newcommand\n{{\mathfrak{n}}}
\newcommand\h{{\mathfrak{h}}}

\newcommand{\lmax}[1]{\mathcal L_{\text{max}}(#1)}
\newcommand{\lmin}[1]{\mathcal L_{\text{min}}(#1)}
\newcommand{\lma}[1]{\mathcal L_{\text{max}}#1}
\newcommand{\lmi}[1]{\mathcal L_{\text{min}}#1}

\newtheorem{thm}{Theorem}[section]
\newtheorem*{theo}{Theorem}
\newtheorem{prop}[thm]{Proposition}
\newtheorem{defn}[thm]{Definition}
\newtheorem{cor}[thm]{Corollary}

\theoremstyle{rem}
\newtheorem{rem}[thm]{Remark}
\theoremstyle{exam}

\begin{document}

\title[Extending Representations]{Extending Representations of $\sl(2)$ to \\ Witt and Virasoro algebras}

\author[F. J. Plaza Mart\'{\i}n]{F. J. Plaza Mart\'{\i}n}
\author[C. Tejero Prieto]{C. Tejero Prieto}

\address{Departamento de Matem\'aticas, Universidad de
Salamanca,  Plaza de la Merced 1-4
        \\
        37008 Salamanca. Spain.
        \\
         Tel: +34 923294460. Fax: +34 923294583}
\address{IUFFYM. Instituto Universitario de F\'{\i}sica Fundamental y Matem\'aticas, Universidad de Salamanca, Plaza de la Merced 1\\ 37008 Salamanca. Spain.}
\thanks{
\indent This work is supported by the research contracts MTM2013-45935-P and
MTM2012-32342 of Ministerio de Ciencia e Innovaci\'{o}n, Spain.  \\
}
\email{fplaza@usal.es}\email{carlost@usal.es}

\begin{abstract}
We study when an $\sl(2)$-repre\-sen\-tation extends to a representation of the Witt and Virasoro algebras. We give a criterion for extendability and apply it to certain classes of weight $\sl(2)$-modules. For all simple weight $\sl(2)$-modules and those in any of the abelian Krull-Schmidt categories of weight modules whose unique simple object is a dense module, we fully characterize which ones admit extensions, and we obtain explicit expressions for all of them. We also give partial results in the same direction for the abelian categories of weight modules which have two and three simple objects.
\end{abstract}

\maketitle

\section{Introduction}

The Virasoro algebra, $\Vir$, is the Lie algebra generated by $\{\{L_i\vert i\in {\mathbb Z}\}, C\}$ as a ${\mathbb C}$-vector space endowed with the bracket
	$$
	\begin{aligned}
	& [ L_i, L_j]\,:=\, (i-j) L_{i+j} + \frac1{12}(i^3-i)\delta_{i+j,0}\, C 
	\\
	& [\Vir, C]\,:=\,0 
	\end{aligned}
	$$
The centerless Virasoro algebra or Witt algebra $\Witt$ is the quotient of $\Vir$ by its center and we identify it with the subset generated by  $\{L_i\vert i\in {\mathbb Z}\}$. Among the subalgebras of both $\Witt$ and $\Vir$ are: $\Witt_{>}$, generated by $\{L_i\vert i\geq -1\}$, $\Witt_{<}$, generated by $\{L_i\vert i\leq 1\}$,  and the subalgebra generated by $\{L_{-1},L_0,L_1\}$  which is isomorphic to $\sl(2)$. Thus, we have an injective map
	$$
	\iota:\sl(2)\hookrightarrow\Vir,
	$$ 
and  if  we choose a Chevalley basis $\{e, f , h\}$ of $\sl(2)$, then $\iota$  sends $f$ to $L_{-1}$, $h$ to $-2L_0$ and $e$ to $-L_1$. 
These Lie algebras have appeared in a variety of problems. The study of $\sl(2)$-representations is the cornerstone of the representation theory of finite dimensional Lie algebras and those of $\Vir$ have also been extensively studied (see, for instance, \cite{Iohara,Kac-Raina,Mathieu} and the references therein), playing a relevant role in mathematical physics (e.g. conformal field theory \cite{Dubrovin,Eguchi}, the geometric Langlands program \cite{Frenkel}, etc.). In addition there  is a long list of problems involving the representations of the algebra $\Witt_{>}$. Its relation with the Virasoro constraints (within the context of the Virasoro conjecture \cite{Eguchi} as well as 2D TFT \cite{Dubrovin}) is one of the most well known and studied. Other instances appear in knot theory (\cite{Dijkgraaf})  and Eynard-Orantin theory (\cite{Eynard}). 

Thus, it is natural to enquire how the map $\iota$ relates the representation theories of $\Vir$ and $\sl(2)$. More precisely, given a $\C$-vector space $V$, the map $\iota$ allows us to consider a restriction map between representations:
    $$ 
    \iota^* :\Hom_{\text{\bf Lie-Alg}}(\Vir,\End(V)) \, \longrightarrow\, \Hom_{\text{\bf Lie-Alg}}(\sl(2),\End(V)).
    $$ 
    
Note that we also have the following inclusions:

$$
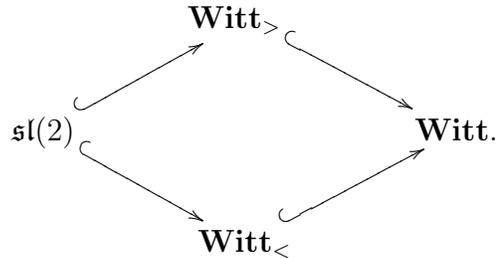
\xymatrix{  & \quad \Witt_>\quad\ \ \ar@{^{(}->}[dr] &  & \\
\sl(2) \ar@{^{(}->} [ur] \ar@{^{(}->} [dr] & &\Witt.\\
 & \Witt_< \ar@{^{(}->} [ur] &  & }$$
 
 

In order to accomplish our goal of relating the representation theory of $\sl(2)$ to that of $\Vir$ we proceed to study the extension problem step by step along the above chain of inclusions. Thus, we begin by studying the relationship between the representations of $\sl(2)$ and those of $\Witt_>$ and $\Witt_<$. These two algebras are exchanged by the Chevalley involution, therefore, it is enough to prove the extendability results for one of them and the analogous results  for the other algebra follow automatically. Once we have established the relationship between the representation theory of $\sl(2)$ and that of $\Witt_>$ and $\Witt_<$, we can proceed to the next step in the chain of inclusions. Given extensions of an $\sl(2)$-representation to $\Witt_>$ and $\Witt_<$, we obtain an extension to $\Witt$ if and only if  they satisfy a certain compatibility condition. In the final step we analyze the extension problem from $\Witt$-modules to $\Vir$-modules. This step by step procedure accounts for this somewhat lengthy exposition.
 
 We believe that this strategy for studying the extension problem step by step along the chain of Lie algebra inclusions between Witt and Virasoro algebras will help to clarify the extension process as well as the several obstructions which arise.

As far as we know, in the literature there are only partial answers to the extension problem  for $\sl(2)$-representa\-tions in the space of differential operators. Juriev (\cite{Juriev1,Juriev2,Juriev3}) provides some constructions of $\Vir$-modules  for the case of Verma $\sl(2)$-modules, where $\sl(2)$ acts by differential operators. We can also mention previous work by the first author on the classification of embeddings of $\Witt$ into the Lie algebra of first order differential operators, $\operatorname{Diff}^1(\C((z)))$ (\cite{Plaza}). The case of first order differential operators on $\mathbb R^3$ is treated by Zhdanov in \cite{Zhdanov}.  Dong, Lin and Mason  in \cite{Dong-Lin} deal with the \emph{opposite} problem; namely, they extract information from a vertex operator algebra by studying its properties as an $\sl(2)$-module. 

Our first main result  gives a characterization of whether  an arbitrary $\sl(2)$-module $V$ admits a compatible $\Witt_>$-module structure; equivalently, we characterize when a representation $\sigma:\sl(2)\to \End(V)$ can be extended to a Lie algebra representation $\Phi:\Witt_>\to \End(V)$. This is the content of Theorem~\ref{thm:Witt>equivalent}, that  can be summarized as follows:

\begin{theo}
Let $V$ be an $\sl(2)$-module defined by a representation $\sigma$ and let  $T\in {\End}( V)$. Let
		$$
		\Phi_>: \Witt_> \, \longrightarrow \,  {\End}( V)
		$$
be the linear map sending $L_{-1}$ to $\sigma(e)$, $L_0$ to $-\frac12 \sigma(h)$, $L_1$ to $-\sigma(e)$ and $L_{2+i}$ to $\frac1{i!}{\ad(\sigma(e))^i}(T)$ for $\geq 0$.
		
Then, $\Phi_>$ is a Lie algebra representation if and only if  the following conditions are fulfilled
	\begin{equation}\label{eq:IntroT}
	\begin{gathered}
	\ad(\sigma(f))(T) \,=\, 3\sigma(e)
	\\ 
	\ad(\sigma(h))(T) \,=\, 4T
	\end{gathered}
	\end{equation}
and there exists $N\in{\mathbb N}$ such that for all $k\geq N$ there holds
	\begin{equation}\label{eq:IntroRedRelation}
	[T,\frac1{(2k-1)!}{\ad(\sigma(e))^{2k-1}}(T)]
	\,=\,-\frac{2k-1}{(2k+1)!}\ad(\sigma(e))^{2k+1}(T)
	\, .
	\end{equation}
\end{theo}

The essential tool that we use for achieving this goal is the theory of partial Lie algebras introduced by Mathieu in \cite{Mathieu}.

We also show that no finite subset of (\ref{eq:IntroRedRelation}) can characterize whether an arbitrary $\sl(2)$-module $(V,\sigma)$ admits such an extension (Theorem~\ref{thm:NoFiniteNumber}). Nevertheless, for a fixed $V$, such a finite subset may exist. Indeed, as a consequence of \S4, it follows that relations (\ref{eq:IntroT}) together with those in (\ref{eq:IntroRedRelation})   corresponding to $k=1$, suffice for the case of infinite dimensional simple weight $\sl(2)$-modules. See also Theorem~\ref{thm:CasimirR23forfirstCategory} for a application of this result to weight $\sl(2)$-modules. 

We also prove a similar result for the algebra $\Witt_<$, showing that the extension of an $\sl(2)$-representation $\sigma$ to $\Phi_<:\Witt_<\to \End(V)$ depends on the existence of an $S\in\End(V)$ fulfilling properties analogous to (\ref{eq:IntroT}) and (\ref{eq:IntroRedRelation}). This is the content of Theorem~\ref{thm:Witt<equivalent}. 

Our second main result is concerned with the extension problem for the $\Vir$ algebra. It is not difficult to check that the extension problem for the Virasoro algebra can be stated in terms of an extension to $\Witt_>$ and another extension to $\Witt_<$, together with a {\emph compatibility constraint} between these two data. Indeed, in Theorem~\ref{thm:LieAlgVirasoro2} we show:

\begin{theo}
Let $(V,\sigma)$ be a non trivial $\sl(2)$-module. Let $T$ (resp. $S$) be an endomorphism of $V$ defining and extension $\Phi_>$ (resp. $\Phi_<$) of $\sigma$ to $\Witt_>$ (resp. $\Witt_<$).  There is a representation of Lie algebras
	$$
	\Phi:\Vir\longrightarrow \End(V)
	$$
such that  $\Phi\vert_{\Witt_<} =\Phi_<$ and  $\Phi\vert_{\Witt_>} =\Phi_>$ if and only if $$K= 4\sigma(h)-2[S,T]\in\End(V),$$ commutes with $T$ and $S$. Moreover,  in this case the central element $C\in\Vir$ is represented by $K$; that is,  $\Phi(C)= K$. 
\end{theo}


In the final section, the above results are applied to compare infinite dimensional weight $\sl(2)$-modules with Harish-Chandra $\Witt_>$-modules and $\Vir$-modules (\cite{Iohara, Mathieu}). For the simple weight $\sl(2)$-modules and those belonging to  the abelian Krull-Schmidt categories of weight modules having a unique simple object which is a dense module, we characterize those admitting extensions obtaining explicit expressions for all of them, giving also partial results for the abelian categories of weight modules which have two and three simple objects. We also show that there are cases in which this construction yields a fully faithful functor from a certain category of weight $\sl(2)$-modules to the category of $\Witt_>$-modules (Theorem~\ref{thm:categorydenseW>}). However, one can not expect in general the existence of a $\Witt_>$-module structure on an arbitrary infinite dimensional weight $\sl(2)$-module, see subsection~\ref{subset:counter}. Further research is needed in order to characterize those $\sl(2)$-modules. Besides this, in a forthcoming paper we study the extension problem for the case of simple non-weight modules.

Among the possible future applications of our results, we believe they could help in the study of affine Kac-Moody algebras and vertex operator algebras as in \cite{Frenkel} (see also \cite{Plaza} for a relation between $\Vir$-modules and opers). In particular, it would allow one to relate the corresponding geometric invariant theories for the $\Witt$ and $\Vir$ algebras with the invariant theory for $\sl(2)$-modules. On the other hand, since our results relate Verma $\Witt_>$-modules and Verma $\sl(2)$-modules, problems similar to the Virasoro conjecture could profit from the combination of our results and the classical representation theory of $\sl(2)$. In particular, the interplay between KP-like hierarchies and Virasoro constraints could be examined following the ideas of~\cite{PlazaArXiv}, where it is shown that a solution of the string and the dilation equations solves automatically all the Virasoro constraints.

\section{Representations of the Witt algebras}


\subsection{$\Witt$ and $\sl(2)$}

Let $\sl(2)$ be the Lie algebra of the Lie group $\operatorname{SL}(2,\C)$. Recall that  $\sl(2)$ is a simple Lie algebra and a Chevalley basis for it consists of a basis $\{e,f,h\}$  satisfying the commutation relations:
    \begin{equation}\label{eq:basis-sl2}
    [e,f]=h\; , \qquad
    [h,e]=2e \; , \qquad
    [h,f]=-2f\; .
    \end{equation}
We thus consider the following triangular decomposition into Lie subalgebras:
	$$
	\sl(2)\, =\, \n^{-}\oplus \h \oplus\n^{+},
	$$
where:
	$$
	 \n^{-} \, =\, <\! f \!> \qquad \h  \, =\, <\! h \!> \qquad \n^{+}  \, =\, <\! e \!>.
	 $$

%


Let $\Witt$ be the Witt algebra; that is, the $\C$-vector space
with basis $\{L_k\}_{k\in\Z}$ endowed with the Lie bracket
$[L_i,L_j]=(i-j) L_{i+j}$, and let us introduce the following subalgebras:
	$$
	\begin{gathered}
	\Witt_{>} \,:=\, <\! \{L_k\}_{k\geq -1} \!>
	\\
	\Witt_{<} \,:=\, <\! \{L_k\}_{k\leq 1} \!>
	\end{gathered}
	$$
Notice that $\Witt$ is a $\Z$-graded Lie algebra by declaring $\deg(L_k)=k$, that is $\Witt_k=<\! L_k \!>$. It is clear that $\Witt_{>}$, $\Witt_{<}$ are $\Z$-graded Lie subalgebras of $\Witt$ and they are interchanged by the Chevalley involution $\Theta\colon \Witt\to \Witt$ defined by $\Theta(L_i)=(-1)^{i+1}L_{-i}$.
Observe that, once a Chevalley basis for $\sl(2)$ has been chosen, there is a natural embedding:
    \begin{equation}\label{eq:emb-sl2-W+}
    \iota: \sl(2)\,\hookrightarrow\, \Witt
    \end{equation}
that sends $f$ to $L_{-1}$, $h$ to $-2L_0$, and $e$ to $-L_1$ and induces an isomorphism of $\sl(2)$ with $\Witt_{<}\cap\Witt_{>}$.

Given a $\C$-vector space, $V$, the embedding (\ref{eq:emb-sl2-W+}) allows us to consider a restriction map between representations:
    $$ 
    \iota^* :\Hom_{\text{Lie-Alg}}(\Witt,\End(V)) \, \longrightarrow\, \Hom_{\text{Lie-Alg}}(\sl(2),\End(V))
    $$ 
and, analogously, a map $\iota^*_{>}$ for $\Witt_{>}$ and $\iota^*_{<}$ for  $\Witt_{<}$.

Forthcoming sections aim at giving necessary and sufficient conditions for an $\sl(2)$ representation $\sigma$ to lie in the image of $\iota^*$ (resp. $\iota^*_{>}$, $\iota^*_{<}$); or, what is tantamount, to characterize $\sl(2)$-modules admitting a compatible structure of $\Witt$-modules (resp. $\Witt_{>}$, $\Witt_{<}$).

\begin{rem}\label{rem:RepreAreInfDim}
 It is well known that $\Witt$, $\Witt_{>}$ and $\Witt_{<}$ are simple Lie algebras, see \cite{patera}. Therefore their non trivial representations are faithful. Accordingly, if  a non trivial $\sl(2)$ representation $\sigma$ on $V$ lies in the image of $\iota^*$ (resp. $\iota^*_{>}$, $\iota^*_{<}$), then $V$ is infinite dimensional.
\end{rem}

\subsection{$\Witt$ and partial Lie algebras}\label{subsec:WittandPartial}
 Notice that if we have an $\sl(2)$ representation $\sigma\in \Hom_{\text{Lie-Alg}}(\sl(2),\End(V))$ which comes from a representation $\rho\in \Hom_{\text{Lie-Alg}}(\Witt,\End(V))$, then the restriction of $\rho$ to $\langle L_{-2},L_{-1},L_0,L_1,L_2\rangle$ defines a representation of this partial Lie algebra which extends $\sigma$. Therefore in order to analyze if an $\sl(2)$ representation extends to a $\Witt$ representation we are naturally led to study the representations of this partial Lie algebra. We shall show below that any representation of $\Witt$ is completely determined by the corresponding partial Lie algebra representation. Moreover, we will prove that the problem of deciding whether an $\sl(2)$ representation extends to a $\Witt$ representation can be completely solved in terms of the partial Lie algebra.

There is a well known relationship between the category $\text{\bf LieAlg}_\bullet$ of $\Z$-graded Lie algebras and the category $\text{\bf LieAlg}\,_d^e$ of partial Lie algebras of size $(d,e)$ studied by Kac   \cite{Kac} for the particular case of local Lie algebras and by Mathieu \cite{Mathieu} in full generality, see also the detailed exposition in \cite[Section 2.2]{Iohara}. Recall that an object of $\text{\bf LieAlg}\,_d^e$ is a vector space $\bigoplus_{i=d}^e\g_i$ with a degree given by $\operatorname{deg}(a)=i$ for $a\in \g_i$ and endowed with bilinear maps $\g_i \times \g_j\to \g_{i+j}$ satisfying the properties of the Lie bracket whenever $i,j,i+j$ lie in the range $(d,e)$. Indeed, there is a natural truncation functor 
$$\text{Par}_d^e\colon \text{\bf LieAlg}_\bullet\to \text{\bf LieAlg}\,_d^e
$$ that maps a graded Lie algebra  $\g=\bigoplus_{i\in\Z}\g_i$ to the partial Lie algebra $\text{Par}_d^e\g=\bigoplus_{i=d}^e\g_i$ which is called the partial part of $\g$ of size $(d,e)$. Moreover, among the graded Lie algebras with a given partial part $\Lambda$ of size $(d,e)$ there is a maximal one $\lmax{\Lambda}$ and a minimal one  $\lmin{\Lambda}$, both of them generated by $\Lambda$, and characterized by the following universal properties. 

Given an object $\Lambda$ in $\text{\bf LieAlg}\,_d^e$, the algebra $\lmin{\Lambda}$ is characterized as the object of  $\text{\bf LieAlg}_\bullet$ whose partial part of size $(d,e)$ is $\Lambda$ and such that  for any graded Lie algebra $\g$, generated by $\text{Par}_d^e\g$, and any surjective morphism of partial Lie algebras $\psi\colon \text{Par}_d^e\g\to \Lambda$, there exists a unique morphism of graded Lie algebras $\Psi\colon \g\to\lmin{\Lambda}$ whose restriction to $\text{Par}_d^e\g$ is $\psi$. 

On the other hand, $\lmax{\Lambda}$, together with the inclusion $\Lambda \hookrightarrow \lmax{\Lambda}$, enjoys the following universal property
	$$
	 \Hom_{\text{\bf LieAlg}_\bullet}(\lmax{\Lambda}, \g) 
	 \,\simeq \, 
	 \Hom_{\text{\bf LieAlg}\,_d^e}(\Lambda, \text{Par}_d^e\g)
	 $$
for any graded Lie algebra $\g$.

These constructions give us functors $\lma{}\colon \text{\bf LieAlg}\,_d^e \to \text{\bf LieAlg}_\bullet$ and  $\lmi{}\colon \text{\bf LieAlg}\,_d^e \to \text{\bf LieAlg}_\bullet$ such that $\lma$ is left adjoint to $\text{Par}_d^e$. Moreover, the above properties yield a canonical map $\lmax{\Lambda}\to \lmin{\Lambda}$ which is surjective. Combining the above facts, the following result is straightforward

\begin{prop}\label{prop:LmaxgLmin}
Let $\Lambda$ be a partial Lie algebra. Let $\g$ be a graded Lie algebra generated by its partial part, $\text{Par}_d^e\g$, and let $\psi$ be an isomorphism $\Lambda\overset\sim\to \text{Par}_d^e\g$. 

The canonical surjective homomorphism $\lmax{\Lambda}\twoheadrightarrow \lmin{\Lambda}$ factors as follows
	$$
	\xymatrix{ 
	\lmax{\Lambda}  \ar@{->>}[rr] \ar@{->>}[rd] & & \lmin{\Lambda}
	\\
	&  \g \ar@{->>}[ur] .}$$ 
\end{prop}

%
%
%
%
%
%

If we write $\Lambda^-=\bigoplus_{d\leq i<0}\Lambda_i$, $\Lambda^+=\bigoplus_{0<i\leq e}\Lambda_i$ and regard them as partial Lie algebras of size $(d,-1)$ and $(1,e)$, respectively, then we get the following triangular decomposition  into graded Lie subalgebras 
$$\lmax{\Lambda}=\lmax{\Lambda^-}\oplus\Lambda_0\oplus \lmax{\Lambda^+}.
$$  Moreover, if $\Lambda$ is any partial Lie algebra of size $(d,e)$ with $-2\leq d, e\leq 2$ then one has $\lmax{\Lambda^\pm}=\mathcal F(\Lambda^\pm)$ where $\mathcal F(\Lambda^\pm)$ denotes the free Lie algebra generated by $\Lambda^\pm$, see for instance \cite[Proposition 2.2, Lemma 2.6]{Iohara}. Hence, in this case the above triangular decomposition reduces to 
$$\lmax{\Lambda}=\F(\Lambda^-)\oplus\Lambda_0\oplus\F(\Lambda^+).$$ 


In particular we are interested in the partial Lie algebra  $\Gamma$ of size $(-2,2)$, such that $\Gamma_i=  <  x_i  >$ for $-2\leq i\leq 2$ and 
$$[x_i,x_j]=(i-j)x_{i+j},
$$ for any $i, j$ such that $-2\leq i, j, i+j\leq 2$. We also consider its partial Lie subalgebras $\Gamma_<=\bigoplus_{i\leq 1} \Gamma_i$, $\Gamma_>=\bigoplus_{i\geq -1} \Gamma_i$ of size $(-2,1)$ and $(-1,2)$, respectively. They are interchanged by the involution $\theta\colon \Gamma\to\Gamma$ defined by $\theta(x_i)=(-1)^{i+1}x_{-1}$. Their relationship with the Witt algebras is as follows 
	{\small 
	\begin{equation}\label{eq:partialWitt}
	 \Par_{-2}^2\Witt \simeq \Gamma ,\quad 
	 \Par_{-1}^2\Witt_> \simeq \Gamma_>  ,\quad 
	 \Par_{-2}^1\Witt_< \simeq \Gamma_< ,
	\end{equation}} 
where in all cases the isomorphism is obtained by mapping $x_i$ to $L_i$. If no confusion arises,  $\Gamma$, $\Gamma_<$ and $\Gamma_>$ will be called the partial Witt algebras. Notice that all these partial algebras contain the subalgebra $\Gamma_<\cap \Gamma_> = \langle  x_{-1},x_0,x_1 \rangle$ and 
sending $x_{-1}$ to $f$, $x_0$ to $-\frac12 h$ and $x_1$ to $-e$ we obtain an isomorphism
	\begin{equation}\label{eq:Sigmasl(2)}
	\Sigma:=\langle  x_{-1},x_0,x_1 \rangle = \Gamma_<\cap \Gamma_>
	\,\overset{\sim}\longrightarrow \, \sl(2).
	\end{equation}

\begin{rem} In what follows when we consider any of the Witt algebras $\Witt, \Witt_>, \Witt_<$, or the partial Witt algebras $\Gamma,\Gamma_>,\Gamma_<$, we denote by $\Pa$ the functor $\Par_{-2}^{2}, \Par_{-1}^{2}, \Par_{-2}^{1}$, respectively.
\end{rem}

\begin{prop}\label{prop:pariso} There are natural isomorphisms relating the Witt Lie algebras and their partial counterparts:

$$\Witt \simeq \lmin{\Gamma},\quad \Witt_>\simeq\lmin{\Gamma_>},\quad \Witt_<\simeq \lmin{\Gamma_<}.$$
\end{prop}

\begin{proof} Let us denote by $\g$ any of the graded Lie algebras $\Witt$, $\Witt_>$, $\Witt_<$ and let $\Lambda$ denote the partial Lie algebras $\Gamma, \Gamma_>, \Gamma_<$, respectively. In all the cases one has that $\g$ is a simple graded Lie algebra generated by $\Pa\!\g$ and there is an isomorphism $\psi\colon \Pa\!\g \simeq \Lambda$ (see equation~(\ref{eq:partialWitt})). Therefore, by Proposition \ref{prop:LmaxgLmin}, we obtain a surjective  morphism of graded Lie algebras 
	$$
	\xymatrix{
	\g \ar@{->>}[r]^-{\Psi} & \lmin{\Lambda}}
	$$
Since $\g$ is simple and $\Psi$ is non-zero, the ideal  $\Ker\Psi$ must be $0$ and, consequently, $\Psi$ is an isomorphism. 
\end{proof}

As above, let $\g$ be any of the graded Lie algebras $\Witt$, $\Witt_>$, $\Witt_<$ and  let $\Lambda$ denote the corresponding partial Lie algebra $\Gamma, \Gamma_>, \Gamma_<$. The isomorphisms given in equation~(\ref{eq:partialWitt}) are explicitly described as
	$$
	\begin{aligned}
	\psi: \Pa\!\g &\,\overset\sim\longrightarrow\, \Lambda \\
	x_i &\, \longmapsto \, L_i .
	\end{aligned}
	$$
Having in mind the universal property of $\lmax{\Lambda}$ and Proposition  \ref{prop:LmaxgLmin}, we obtain a short exact sequence of graded Lie algebras 
	\begin{equation}\label{eq:IgLmaxg}
	0 \longrightarrow I(\g):= \Ker\pi 
	\longrightarrow \lmax{\Lambda} \overset{\pi} \longrightarrow \g \longrightarrow 0.
	\end{equation}

%

\begin{defn} We say that $I(\g)$ is the ideal associated to the graded Lie algebra $\g$.
\end{defn}

%

\begin{prop}\label{prop:triang-desc} One has the following triangular decompositions into graded Lie subalgebras \begin{align*}
\lmax{\Gamma} &=\F(\langle  x_{-2},x_{-1} \rangle) \oplus \langle  x_0 \rangle \oplus \F(\langle  x_1,x_2 \rangle),\\
\lmax{\Gamma_>} &=\langle  x_{-1} \rangle \oplus \langle  x_0 \rangle \oplus \F(\langle  x_1,x_2 \rangle),\\
\lmax{\Gamma_<} &=\F(\langle  x_{-2},x_{-1} \rangle) \oplus \langle  x_0 \rangle \oplus \langle  x_1 \rangle.\end{align*}
Moreover, we have the following decompositions into graded Lie subalgebras \begin{align*} \F(\langle  x_{-2},x_{-1} \rangle) &=\F(\langle  x_{-(2+i)} \rangle_{i\geq 0})\oplus \langle  x_{-1} \rangle,\\
\F(\langle  x_{1},x_{2} \rangle) &=\langle  x_{1} \rangle\oplus \F(\langle  x_{2+i} \rangle_{i\geq 0}),\end{align*} where $x_{-(2+i)}:=\frac{1}{i!}(\ad_{x_{-1}})^i(x_{-2})$, $x_{2+i}:=\frac{1}{i!}(\ad_{-x_{1}})^i(x_{2})$.
\end{prop}

\begin{proof} The triangular decompositions are particular cases of the general result mentioned above for partial Lie algebras of size $(d,e)$ with $-2\leq d, e\leq 2$. The other decompositions follow from the Elimination Theorem for free Lie algebras, see \cite[Theorem 0.2]{Reutenauer}.
\end{proof}

Let us recall that the support of a $\Z$-graded Lie algebra is the set 
$$\supp(\g)=\{i\in\Z\colon \g_i\neq 0\}.
$$

\begin{cor}\label{cor:idealstructure} The ideal $I(\Witt)$ associated to the Witt algebra has the following decomposition into graded ideals
$$ I(\Witt)=I^-(\Witt)\oplus I^+(\Witt),$$ where $I^-(\Witt)\subset \F(\langle  x_{-(2+i)} \rangle_{i\geq 0})$,  $I^+(\Witt)\subset \F(\langle  x_{2+i} \rangle_{i\geq 0})$ and $$\supp(I^-(\Witt))\subset  \Z_{\leq -5},\quad \supp(I^+(\Witt))\subset  \Z_{\geq 5}.$$

Moreover one has $
I(\Witt_>)= I^+(\Witt)$, $I(\Witt_<)= I^-(\Witt)$.
\end{cor}

\begin{proof} Since the quotient map $\lmax{\Gamma}\xrightarrow{\pi}\Witt$ is induced by mapping $x_i\in\Gamma$ to $L_i$, it follows that $\pi(x_i)=L_i$ for every $i\in\Z$. Now taking into account the triangular decomposition of $\lmax{\Gamma}$ given in Proposition \ref{prop:triang-desc} and using a Hall basis for the free Lie algebras that appear in the decomposition, one easily checks  that $\lmax{\Gamma}_i=\langle x_i\rangle$ for every $-4\leq i\leq 4$. The other statements  follow now straightforwardly.

\end{proof}

\begin{defn}\label{defn:standard-relations}
The ideals of standard relations for the Witt algebras are:
\begin{enumerate}
\item $R(\Witt_>)$ is the ideal of $\lmax{\Gamma_>}$ generated by 
$$r_{2+i,2+j}:=[x_{2+i},x_{2+j}]-(i-j)x_{2+i+j+2}\in \F(\langle  x_{2+k} \rangle_{k\geq 0}),
$$for every $i,j\geq  0$, $i<j$.

\item $R(\Witt_<)$ is the ideal of $\lmax{\Gamma_<}$ generated by 
$$r_{-2+i,-2+j}:=[x_{-2+i},x_{-2+j}]-(i-j)x_{-2+i+j-2}\in \F(\langle  x_{-(2+k)} \rangle_{k\geq 0}),
$$ for every $i,j\leq 0$, $i<j$.

\item $R(\Witt)$ is the ideal of $\lmax{\Gamma}$ given by 
$$R(\Witt)=R(\Witt_<)\oplus R(\Witt_>).
$$\end{enumerate}

\end{defn}

\begin{prop}\label{prop:standard-relations} The  Witt algebras satisfy 
$$R(\Witt_>)=I(\Witt_>),\ R(\Witt_<)=I(\Witt_<),\ R(\Witt)=I(\Witt).
$$
\end{prop}

\begin{proof} Let $\g$ be any of the graded Lie algebras $\Witt_>,\Witt_<, \Witt$ and  let $\Lambda$ be the partial Lie algebras $\Gamma_>,\Gamma_<, \Gamma$, respectively. Since $\pi(x_i)=L_i$ for every $i\in\Z$ it follows that  that $R(\g)\subset I(\g)$. Therefore there exists a commutative diagram of exact sequences

$$\xymatrix{0 \ar[r] & R(\g)\ar[d]\ar[r] & \lmax{\Lambda}\ar[d]^{\operatorname{Id}} \ar[r]^(.4)p & \lmax{\Lambda}/R(\g) \ar[d]^\Psi \ar[r] & 0\\ 
0 \ar[r] & I(\g)\ar[r] & \lmax{\Lambda} \ar[r]^\pi & \g    \ar[r] & 0}
$$ and the snake Lemma yields  $I(\g)/R(\g)\simeq \Ker\Psi$. Thus it is enough to show that $\Psi$ is injective. Let us denote $\bar x_i:=p(x_i)$. From the commutative diagram we get $\Psi(\bar x_i)=L_i$. 

Let us consider now the case of $\Witt_>$. From Proposition \ref{prop:triang-desc} and Corollary \ref{cor:idealstructure} we get 
$$\lmax{\Gamma_>}/R(\Witt_>)=\langle  \bar x_{-1} \rangle \oplus \langle  \bar x_0 \rangle \oplus \F(\langle  x_1,x_2 \rangle)/R(\Witt_>).
$$ The standard relations imply that for any $i\geq 0, j\geq 0$ we have $[\bar x_{2+i},\bar x_{2+j}]=(i-j)\bar x_{2+i+j+2}$. Therefore the image under $p$ of any Lie monomial in $\F(\langle  x_1,x_2 \rangle)$ reduces to some $\bar x_{2+k}$. Hence we have 
$$\lmax{\Gamma_>}/R(\Witt_>)=\bigoplus_{i\geq -1}\langle\bar x_i \rangle,
$$ and since $\Psi(\bar x_i)=L_i$ it is clear that $\Psi$ is injective. The other cases are proved in a similar way.
\end{proof}

\begin{defn} Let $\Lambda$ be a partial Lie algebra of size $(d,e)$. A linear representation of $\Lambda$ on a complex vector space $V$, or a $\Lambda$-module structure on $V$, is a linear map $\phi\colon \Lambda \to\End(V)$ such that 
$$\phi([x_i,x_j])=[\phi(x_i),\phi(x_j)],
$$ for any $i, j$ verifying $d\leq i, j, i+j\leq e$ and any $x_i\in \Lambda_i$, $x_j\in \Lambda_j$.
\end{defn}


\begin{defn} Let $\Lambda$ be any of the partial Witt algebras. Given an $\sl(2)$-module $V$ defined by a representation $\sigma\colon \sl(2)\to\End(V)$ we say that a $\Lambda$-module structure defined by a representation $\phi\colon \Lambda \to\End(V)$ is compatible with the given $\sl(2)$-module structure if its restriction to $\Sigma\simeq \sl(2)$ is $\sigma$; that is $\phi_{|\Sigma}=\sigma$ through the identification of equation~(\ref{eq:Sigmasl(2)}).
\end{defn}

\begin{prop}\label{prop:nontrivial} If $\phi\colon \Lambda \to\End(V)$ is a non trivial representation of any of the partial Witt algebras, then either $\phi$ is injective or the induced $\sl(2)$-representation  $\sigma\colon \Sigma \to \End(V)$ is trivial and $\Ker\phi=\Sigma$. Therefore if the induced $\sl(2)$-module structure on $V$ is non trivial then   $\phi$ is necessarily injective.

\end{prop}

\begin{proof} Let us prove it for $\Gamma$. If $\phi(x_2)=0$ then 
$$0=[\phi(x_2),\phi(x_{-2})]=\phi([x_2,x_{-2}])=4\phi(x_0).
$$ Therefore for any $-2\leq i\leq 2$ we have $0=[\phi(x_0),\phi(x_i)]=-i\phi(x_i)$ and thus $\phi$ is trivial. A similar argument shows that $\phi$ must also be trivial if $\phi(x_{-2})=0$. Now let us suppose that $T=\sum_{i=-2}^2\lambda_{i}\phi(x_i)=0$, then $(\ad_{-\phi(x_0)})^3T-\ad_{-\phi(x_0)})T=0$ gives $\lambda_{-2}\phi(x_{-2})+\lambda_2\phi(x_2)=0$, and applying $\frac{1}{2}\ad_{\phi(x_0)}$ to this expression we get $\lambda_{-2}\phi(x_{-2})-\lambda_2\phi(x_2)=0$. These two expressions imply $\lambda_{-2}\phi(x_{-2})=\lambda_2\phi(x_2)=0$. Hence if $\phi$ is non trivial then we must have $\lambda_{-2}=\lambda_2=0$. Thus $\Ker\phi\subset \Sigma\simeq\sl(2)$ and the claim follows since $\sl(2)$ is simple. The proof for the other partial Witt algebras is entirely similar.
\end{proof}

Recall that our aim is to characterize those $\sl(2)$-modules $V$ which can be endowed with a compatible structure of $\Witt_>$-module; that is, whether there exists a Lie algebra map $\Witt_>\to \End(V)$ extending the $\sl(2)$-module structure on $V$ defined by a given representation $\sigma\colon \sl(2) \to\End(V)$.

\begin{defn} Let $V$ be an $\sl(2)$-module defined via a representation $\sigma$. For every $ n\in\Z$ we set
$$\End_n(V,\sigma)\,:=\, \{T\in\End(V)\colon [\frac12\sigma(h),T]=n\cdot T\}.
$$ 
%
%
\end{defn}

Notice that if $\sigma\colon \sl(2) \to\End(V)$ is a trivial representation, then $\End_n(V,\sigma)=0$ if $n\neq 0$ and $\End_0(V,\sigma)=\End(V)$.

From now on, $V$ will denote an $\sl(2)$-module defined via a representation $\sigma\colon \sl(2) \to\End(V)$. Recall that $\Lambda$ carries an $\sl(2)$-module structure via the isomorphism (\ref{eq:Sigmasl(2)}).

It is worth pointing out that $\End_n(V,\sigma)$ is the eigenspace of the operator $\frac12\ad(\sigma(h))$ acting on $\End(V)$ corresponding to the eigenvalue $n\in\Z$. It is straightforward to check that the external direct sum 
	$$
	\End_\bullet(V,\sigma)\,:=\, \bigoplus_{n\in\Z}\End_n(V,\sigma)
	$$
can be identified with a subspace of $\End(V)$ which carries a canonical $\Z$-graded Lie algebra structure thanks to the Jacobi identity.  In what follows, if no confusion arises, then $\sigma$ will be omitted; that is, we will simply write  $\End_n(V)$ and $\End_\bullet(V)$.

\begin{prop} 
Let  $\phi\colon \Lambda \to\End(V)$ be a representation of any of the partial Witt algebras  and let $\g$ be the corresponding Witt algebra. 

If the $\sl(2)$-module structure induced on $V$ is non trivial then there is an injection $\phi\colon \Lambda \hookrightarrow \Pa\!\End_\bullet(V)$ and $\supp(\End_\bullet(V))=\supp(\g)$.
\end{prop}

\begin{proof} 
%
Notice that $\phi(x_i)\in\End_i(V)$ for every $-1\leq i\leq 2$. Since the $\sl(2)$-module structure induced on $V$ is non trivial, it follows from Proposition \ref{prop:nontrivial} that   $\End_i(V)\neq 0$ for every $-1\leq i\leq 2$. Let us suppose that $\End_{2+i}(V)=0$ then $0=(\ad_{\phi(x_1)})^i\phi(x_2)$. However since $(\ad_{\phi(x_{-1})})^i(\ad_{\phi(x_1)})^i\phi(x_2)=m\cdot\phi(x_2)$ for a non zero integer $m$, this would imply that $\phi(x_2)=0$. Therefore $\phi$ would not be injective and by  Proposition \ref{prop:nontrivial} this contradicts the assumption that $V$ is a non trivial $\sl(2)$-module. Therefore $\supp(\End_\bullet(V))=\supp(\Witt_>)$. The proof for the other partial Witt algebras is completely similar.
\end{proof}

Now the universal property of $\lmax{\Lambda}$ gives the following key result.

\begin{thm} Let  $\phi\colon \Lambda \to\End(V)$ be a representation of any of the partial Witt algebras  and let $\g$ be the corresponding Witt algebra. 

There is a canonical bijective correspondence
	$$
	\Hom_{\sl(2)-\text{mod}}(\Lambda, \End(V)) \,\overset\sim\longrightarrow\, 
	\Hom_{\text{\bf LieAlg}_\bullet}(\lmax{\Lambda}, \End_\bullet(V))
	$$
where $\psi$ is mapped to $\bar \psi$ if and only if they fit in the commutative diagram
	\begin{equation}\label{eq:LambdaEndbullet}
	\xymatrix{
	\Lambda \ar@{^(->}[d] \ar[r]^{\psi} & 
	\End(V) \\
	 \lmax{\Lambda} \ar[r]^{\bar \psi} &  
	 \End_\bullet(V)\ar@{^(->}[u] }
	 \end{equation}
\end{thm}

\begin{proof} 
Let $\psi:\Lambda\to \End(V)$ be a linear representation of $\Lambda$ on $V$. It is obvious that it takes values in $\End_\bullet(V)$. Applying the universal property of $\lmax{\Lambda}$ to the resulting map $\psi:\Lambda\to\End_\bullet(V)$, we obtain
	$$
	\bar \psi : \lmax{\Lambda} \longrightarrow \End_\bullet(V) .
	$$
It is easy to check that this construction yields the desired bijection. 
\end{proof} 

\begin{cor}\label{cor:extension} 
Let  $\g$ be any of the graded Lie algebras $\Witt$, $\Witt_>$ or $\Witt_<$ and let $\Lambda$ denote the partial Lie algebras $\Gamma, \Gamma_>, \Gamma_<$, respectively.

A representation $\phi\colon \Lambda \to\End(V)$  extends to a representation $\rho\colon \g\to \End_\bullet(V)$  if and only if $\Phi\colon \lmax{\Lambda}\to \End_\bullet(V)$ vanishes on the standard relations of $\g$.  In this case, it holds that $\rho(L_i)=\Phi(x_i)$.
\end{cor}

\begin{proof} 
It follows from the exact sequence~(\ref{eq:IgLmaxg}) and the commutative diagram~(\ref{eq:LambdaEndbullet}). 
%
%
\end{proof}


\begin{thm}\label{thm:reduced-standard-relations} The ideals of standard relations for the Witt algebras satisfy:
\begin{enumerate}
\item $R(\Witt_>)$ is generated by 
$$r_{2,2k+1}=[x_2,x_{2k+1}]+(2k-1)x_{2k+3},
$$ for $k\geq 1$.

\item $R(\Witt_<)$ is generated by 
$$r_{-(2k+1),-2}=[x_{-(2k+1)},x_{-2}]+(2k-1)x_{-(2k+3)},
$$ for $k\geq 1$.

\item $R(\Witt)$ is generated by $r_{2,2k+1}, r_{-(2k+1),-2}$ for $k\geq 1$.\end{enumerate}
\end{thm}

\begin{proof}
Let us prove the statement for $\Witt_>$. By Definition \ref{defn:standard-relations} the ideal $R(\Witt_>)$ is generated by 
$$r_{2+i,2+j}=[x_{2+i},x_{2+j}]-(i-j)x_{2+i+j+2},
$$ for every $i,j\leq 0$, $i<j$. Let $R_n=\langle r_{2+i,2+j}\rangle_{n=i+j+4}$ be the complex vector space generated by those standard relations whose degree, as elements of the $\Z$-graded Lie algebra $\lmax{\Gamma_>}$, is $n$. Recall that $n\geq 5$ by Corollary \ref{cor:idealstructure}. A straightforward computation shows that $d_n:=\dim_\C R_n-1=\left[\frac{n-1}{2}\right]-2$, where $[t]$ denotes the integer part of $t$. Moreover if $r_i^{n}=r_{2+i,2+n-4-i}$ then $B^{(n)}=\{r_i^{(n)}\}_{i=0}^{d_n}$ is a basis of $R_n$. Let us denote $e=\ad_{-x_1}\colon\lmax{\Gamma_>}\to \lmax{\Gamma_>}$. It is clear that $e$ induces a map $e_n\colon R_n\to R_{n+1}$ and an easy computation shows that 
$$e_n\cdot r_i^{(n)}=(n-i-3)\, r_i^{(n+1)}+(i+1)\ r_{i+1}^{(n+1)}.
$$
Let us consider now for every $k\geq 3$ the sequence of linear mappings 
$$R_{2k-1}\xrightarrow{e_{2k-1}} R_{2k}\xrightarrow{e_{2k}} R_{2k+1}
$$ and recall that $\dim_\C R_{2k-1}=\dim_\C R_{2k}=\dim_\C R_{2k+1}-1=k-2$. By considering the matrices of these mappings with respect to the bases defined above one immediately checks that $e_{2k-1}$ is an isomorphism, $e_{2k}$ is injective and $R_{2k+1}=e_{2k}(R_{2k})\oplus\langle r^{2k+1}_0\rangle$. This proves that the ideal $R(\Witt_>)$ is generated by $\{r^{(2k+1)}_0\}_{2k+1\geq 5}=\{r_{2,2k-1}\}_{k\geq 2}=\{r_{2,2k+1}\}_{k\geq 1}$ and so the claim is proved.

The result for $\Witt_<$ can be obtained from the previous one by considering the Chevalley involution $\Theta$. Bearing in mind now that  $R(\Witt)=R(\Witt_<)\oplus R(\Witt_>)$ the proof for $\Witt$ follows immediately.  
\end{proof}

\begin{cor}\label{cor:levelN-reduced-standard-relations} For any integer $N\geq 1$ we have:
\begin{enumerate}
\item $R(\Witt_>)$ is generated by 
$$r_k:=r_{2,2k+1}=[x_2,x_{2k+1}]+(2k-1)x_{2k+3},
$$ for $k\geq N$.

\item $R(\Witt_<)$ is generated by 
$$r_{-k}:=r_{-(2k+1),-2}=[x_{-(2k+1)},x_{-2}]+(2k-1)x_{-(2k+3)},
$$ for $k\geq N$.

\item $R(\Witt)$ is generated by $r_k, r_{-k}$ for $k\geq N$.\end{enumerate}
\end{cor}

\begin{proof} We start with the Witt algebra $\Witt_>$.  Take the linear map $f=\ad_{x_{-1}}\colon\lmax{\Gamma_>}\to \lmax{\Gamma_>}$. It is clear that $f$ restricts to a map $f\colon R_n\to R_{n-1}$ and an easy computation shows that 
$$f\cdot r_i^{(n)}=-(n-i-1)\, r_i^{(n-1)}-(i+3)\ r_{i-1}^{(n-1)}.
$$ In particular one has $f\cdot r_0^{(n)}=-(n-1)\, r_0^{(n-1)}$. Therefore if we consider the ideal $I$ generated by the set $\{r_0^{(2k+3)}=r_{2,2k+1}\}_{k\geq N}\subset R(\Witt_>)$ then 
$$f^{2i}\cdot r_{2,2N+1}=2N\cdot(2N-1)\cdots(2(N-i)+1)r_{2,2(N-i)-1}.
$$ Thus $I$ contains the set $\{r_{2,2k+1}\}_{k\geq 1}$ and by Theorem \ref{thm:reduced-standard-relations} we conclude that $I=R(\Witt_>)$. The proof for $\Witt_<$ can be obtained from the previous one by considering the Chevalley involution $\Theta$. Bearing in mind now that  $R(\Witt)=R(\Witt_<)\oplus R(\Witt_>)$ the result for $\Witt$ follows immediately.

\end{proof}

\begin{defn}\label{defn:reduced-standard-relations}
Each set of generators given in Corollary \ref{cor:levelN-reduced-standard-relations} is called the level $N$ reduced standard relations for the corresponding Witt algebra.
\end{defn}

Taking together Corollaries \ref{cor:extension} and \ref{cor:levelN-reduced-standard-relations} gives the following criterion for extending representations of the partial Witt algebras.

\begin{cor}\label{cor:levelN-standard-relations-extension} A representation $\phi\colon \Lambda \to\End(V)$ of any of the partial Witt algebras extends to a representation $\rho\colon \g\to \End_\bullet(V)$ of the corresponding Witt algebra if and only if $\Phi\colon \lmax{\Lambda}\to \End_\bullet(V)$ vanishes on the level $N\geq 1$ reduced standard relations of $\g$. 
\end{cor}

\begin{rem}\label{rem:clave}
Thus, the space of $\g$-module structures on a vector space $V$, i.e $\Hom_{\text{Lie-Alg}}(\g,\End(V))$, is a subset of the space of $\Lambda$-module structures on $V$, i.e $\homparlie(\Lambda,\End(V))$. In particular the subset $\Hom_{\text{\bf LieAlg}_\bullet}^\sigma(\g,\End_\bullet(V))$ of $\Hom_{\text{\bf LieAlg}_\bullet}(\g,\End(V))$ defined by those $\g$-module structures on $V$ compatible with a given $\sl(2)$-module structure $\sigma$ on $V$ (i.e $\Hom_{\text{\bf LieAlg}_\bullet}^\sigma(\g,\End_\bullet(V))=(\iota^*)^{-1}(\sigma)$) can be identified with a subset of the space $\homparlie^\sigma(\Lambda,\End(V))$ formed by those $\Lambda$-module structures that are compatible with $\sigma$.
\end{rem}

We have seen above that the standard relation $r_k$ of $\Witt_>$ implies the first $(k-1)$ standard relations. We finish this section by showing that the converse statement is not true. That is, $r_k$ does not follow from the first $(k-1)$ standard relations.

\begin{thm}\label{thm:NoFiniteNumber}
For any integer $k\geq 2$ we have:
\begin{enumerate}
\item The ideal $\langle r_1,\ldots,r_{k-1}\rangle $ of $\lmax{\Gamma_>}$ does not contain $r_k$. 
\item The ideal $\langle r_{-1},\ldots,r_{-(k-1)}\rangle $ of $\lmax{\Gamma_<}$ does not contain $r_{-k}$. 
\item The ideal $\langle r_{-(k-1)},\ldots,r_{(k-1)}\rangle $ of $\lmax{\Gamma}$ does not contain either\ $r_{-k}$, or $r_k$.\end{enumerate}
\end{thm}

\begin{proof} We have the decomposition 
$$\lmax{\Gamma_>}=\langle x_1\rangle\oplus\langle x_0\rangle\oplus \langle x_1\rangle\oplus\mathcal F(\langle x_{2+i}\rangle_{i\geq 0})=\Sigma\oplus \mathcal F(\langle x_{2+i}\rangle_{i\geq 0}),
$$ and $\mathcal F(\langle x_{2+i}\rangle_{i\geq 0})$ is a $\Sigma\simeq \sl(2)$-module. Therefore the ideal generated by the first $(k-1)$ standard relations $I=\langle r_1,\ldots,r_{k-1}\rangle $ is contained in $\mathcal F(\langle x_{2+i}\rangle_{i\geq 0})$. An element $Q\in I$ is a Lie polynomial of the form $Q=P_1\cdot r_1+\cdots P_{k-1}\cdot r_{k-1},$ where $P_i\cdot r_i$ denotes a Lie polynomial whose last element is $r_i$, that is a finite sum 
$$P_i\cdot r_i=\sum_{j_1\cdots j_m}\lambda_i^{j_1\cdots j_m}[A_i^{j_1},[A_i^{j_2},\cdots [A_i^{j_m},r_i]\cdots ],
$$ where $A_i^{j_s}\in\lmax{\Gamma_>}$ and  therefore it is itself a finite sum of the form  

$$A_i=a_i^{-1}x_1+a_i^0 x_0+a_i^1 x_1+\sum_{\alpha_1,\ldots,\alpha_n}a_i^{\alpha_1,\ldots,\alpha_n}[x_{\alpha_1},[x_{\alpha_2},\cdots [x_{\alpha_{n-1}},x_{\alpha_n}]\cdots],
$$ with $\alpha_l\geq 2$ for all $l$. Since $\mathcal F(\langle x_{2+i}\rangle_{i\geq 0})$ is a $\Sigma\simeq \sl(2)$-module and the action of $x_0$ on a monomial is by multiplication by a constant, by mean of the Jacobi identity we can always write $P_i\cdot r_i$ as a finite sum of Lie monomials of the form 
$$[A_i^{j_1},[A_i^{j_2},\cdots [A_i^{j_m},e^p\cdot f^q\cdot r_i]\cdots ],
$$ where $A_i^{j_s}=[x_{\alpha_1},[x_{\alpha_2},\cdots [x_{\alpha_{n-1}},x_{\alpha_n}]\cdots]$ with $\alpha_l\geq 2$ for all $l$ and $e=\ad_{-x_1}$, $f=\ad_{x_{-1}}$. We have seen in the proof of Corollary \ref{cor:levelN-reduced-standard-relations} that $\{f^q\cdot r_{k-1}\}_{q\geq 0}$ gives $\{r_1,\ldots, r_{k-1}\}$ and the proof of Theorem \ref{thm:reduced-standard-relations} shows that the terms of $\{e^p\cdot\{r_1,\ldots, r_{k-1}\}\}_{p\geq 0}$ up to order $2k+3$ generate $R_5\oplus\cdots\oplus R_{2k+2}\oplus e(R_{2k+2})$. Therefore an element $Q\in I$ of degree $2k+3$ belongs to 
$$P_{2k-2}\cdot R_5\oplus\cdots\oplus P_2\cdot R_{2k+1}\oplus e\cdot R_{2k+2},
$$ where $P_{2k-1-i}\cdot R_{4+i}$ denotes the homogeneous Lie polynomials of degree $2k+3$ whose last element is in $R_{4+i}$ and its other elements belong to $\mathcal F(\langle x_{2+i}\rangle_{i\geq 0})$. Hence if $\bar Q_i\in P_{2k-1-i}\cdot R_{4+i}$ then it can be written as 
$$\bar Q_i=\sum_{j_1\cdots j_m}\lambda_i^{j_1\cdots j_m\alpha}[A_i^{j_1},[A_i^{j_2},\cdots [A_i^{j_m},Q_i^\alpha]\cdots ],
$$ where $A_i^{j_s}\in\mathcal F(\langle x_{2+i}\rangle_{i\geq 0})$, $ Q_i^\alpha\in R_{4+i}$ and $\deg(A_i^{j_1})+\cdots+\deg(A_i^{j_m})+i+4=2k+3$. Notice that the unique element of $\mathcal F(\langle x_{2+i}\rangle_{i\geq 0})$ that has degree $2$ is $x_2$, thus $P_{2}\cdot R_{2k+1}=[x_2,R_{2k+1}]$. Therefore we can write $Q=Q'+ Q'_1+ Q'_2$ with $Q'\in P_{2k-2}\cdot R_5+\cdots+P_3\cdot R_{2k}$, $Q'_1=[x_2,Q_{2k+1}]$ and $Q'_2=e\cdot Q_{2k+2}$ for certain $Q_{2k+1}\in R_{2k+1}$, $Q_{2k+2}\in R_{2k+2}$. Thus 
$$Q'_1=\sum_{i=0}^{k-2}\lambda_i[x_2,r^{(2k+1)}_i],\quad Q'_2=\sum_{i=0}^{k-2}\mu_i e\cdot r^{(2k+2)}_i.
$$

Let us recall that the free Lie algebra $\mathcal F(\langle x_{2+i}\rangle_{i\geq 0})$ can be realized as the subalgebra of the tensor algebra  $T^\bullet(V)$ generated by $V$, where $V$ is the vector space with basis $\{x_{2+i}\}_{i\geq 0}$. Let us denote by $\pi_s\colon T^\bullet(V)\to T^s(V)$ the natural projection, and notice that a monomial $[x_{\alpha_1},[x_{\alpha_2},\cdots [x_{\alpha_{n-1}},x_{\alpha_n}]\cdots]\in \mathcal F(\langle x_{2+i}\rangle_{i\geq 0})$ thought as an element of $T^\bullet(V)$ belongs to $T^n(V)$. It follows from the above expressions that \begin{align*} \pi_1(P_{2k-2}\cdot R_5) &=0,\quad \pi_2(P_{2k-2}\cdot R_5)\in\langle [x_{2k-2},x_5]\rangle, \\
& \vdots\quad\quad\quad\vdots\\
\pi_1(P_{3}\cdot R_{2k}) &=0,\quad \pi_2(P_{2}\cdot R_{2k})\in\langle [x_3,x_{2k}]\rangle,\\
\pi_1(P_{2}\cdot R_{2k+1}) &=0,\quad \pi_2(P_{2}\cdot R_{2k+1})\in\langle [x_2,x_{2k+1}]\rangle.
\end{align*} On the other hand we have $\pi_1(e\cdot R_{2k+2})\in\langle x_{2k+3}\rangle$ and 
$$\pi_2(e\cdot R_{2k+2})\in\langle [x_2,x_{2k+1}],[x_3,x_{2k}],\ldots, [x_{k+1},x_{k+2}]\rangle.
$$

Now suppose that there exists $Q\in I $ such that $Q=r_k$. Writing as before $Q=Q'+Q'_1+Q'_2$ then we must have $\pi_s(Q)=\pi_s(r_k)$ for every $s\geq 1$. The expressions above show that we must have \begin{align*}
\pi_1(Q'_2) &=\pi_1(r_k)=(2k-1)x_{2k+3},\\ \pi_2(Q') &=0,\\
\pi_2(Q'_1)+\pi_2(Q'_2) &=\pi_2(r_k)=[x_2,x_{2k+1}],\\ \pi_s(Q) &=\pi_s(r_k)=0, \quad \text{if}\
s\geq 3.\end{align*} The equality $\pi_2(Q'_1)+\pi_2(Q'_2) =[x_2,x_{2k+1}]$ is equivalent to the following systems of equations \begin{align*}\sum_{i=0}^{k-2}(2k-2i-3)\lambda_i+(2k+1)\mu_0 &=1,\\
2k\mu_1+\mu_0 &=0,\\
(2k-1)\mu_2+2\mu_1 &=0,\\
&\vdots\\
(k+3)\mu_{k-2}+(k-2)\mu_{k-3}&=0,\\
(k-1)\mu_{k-2} &=0.
\end{align*} Hence $\mu_0=\ldots=\mu_{k-2}=0$, that is $Q'_2=0$. This leads to a contradiction since at the same time we should have $\pi_1(Q'_2)=(2k-1)x_{2k+3}$. Therefore, $Q	\notin I$ and the theorem is proved for $\Gamma_>$.  By mean of the Chevalley automorphism this gives also the result for $\Gamma_<$. A similar argument also works for $\Gamma$.

\end{proof}

\begin{rem} Therefore if a representation $\phi\colon \Lambda \to\End(V)$ of any of the partial Witt algebras is such that $\Phi\colon \lmax{\Lambda}\to \End_\bullet(V)$ vanishes on the first $(k-1)$ standard relations then, in general this does not imply that $\phi$ will extend to a representation $\rho\colon \g\to \End_\bullet(V)$ of the corresponding Witt algebra.\end{rem}

\subsection{Extending to $\Witt_{>}$}

We have seen in Remark \ref{rem:clave} that given an $\sl(2)$-module structure on $V$ defined by $\sigma\in \Hom_{\text{Lie-Alg}}(\sl(2),\End(V))$, the space of $\Witt_>$-module structures on $V$ compatible with $\sigma$ can be identified with a subset of the space $\homparlie^\sigma({\Gamma}_>,\End_\bullet(V))$ of $\Gamma$-module structures on $V$ compatible with $\sigma$.

\begin{prop}\label{prop:Tiffrho}
Let $V$ be an $\sl(2)$-module defined by a representation  $\sigma$.
The space ${\Gamma}_>( V,\sigma)$ of ${\Gamma}_>$-module structures on $V$ compatible with $\sigma$ can be identified with the set ${\End}_{\Gamma_>}( V,\sigma)\subset\End(V)$ formed by those endomorphisms $T\in \End(V)$ which satisfy the equations \begin{equation}\label{eq:g-L_2}
	    	\begin{aligned}
		   & \ad(\sigma(f))(T) =  3\sigma(e), \\
	    	  & \ad(\sigma(h))(T) = 4T,
		  \end{aligned}
	    \end{equation} 
by mapping $\phi\in {\Gamma}_>( V,\sigma)$ to $T=\phi(x_2)\in \End(V) $.

Therefore the space ${\Witt}_>( V,\sigma)$ of ${\Witt}_>$-module structures on $V$ compatible with $\sigma$ can be identified with a subset $\End_{\Witt_>}(V,\sigma)\subset {\End}_{\Gamma_>}( V,\sigma) $.

\end{prop}

\begin{proof} Given $\phi\in {\Gamma}_>( V,\sigma)$, since \begin{equation}\label{eq:defrho-1012}
    \begin{aligned}
    \phi(L_{-1})\,=& \, \sigma(f),
    \\
    \phi(L_0)\,=& \,  -\frac12 \sigma(h),
    \\
    \phi(L_1)\,=& \,  -\sigma(e),
    \\
    \phi(L_2)\,=& \,  T,
    \end{aligned}
    \end{equation} and $\phi$ is a representation of the partial Lie algebra $\Gamma_>$, it is straightforward to check that $T=\phi(L_2)\in \End(V)$ satisfies the equations (\ref{eq:g-L_2}). This proves the first statement. The second one follows from Corollary \ref{cor:levelN-standard-relations-extension}. \end{proof}

\begin{prop} Let $V$ be an $\sl(2)$-module defined by a representation $\sigma$.
Given  $T\in {\End}_>( V,\sigma)$ the corresponding Lie algebra morphism  $\Phi\in \Hom_{\text{\bf LieAlg}_\bullet}(\lmax{\Gamma}_>,\End_\bullet{}(V))$ satisfies
	\begin{equation}\label{eq:rho(L_i)def}
    \Phi(x_i)= \frac1{i-2} \ad(\sigma(e))(\Phi(x_{i-1}))
    \qquad \text{for }i>2
    \;,
    \end{equation} and therefore 
$$\Phi(x_{2+i})=\frac1{i!} \ad(\sigma(e))^i(T)
    \qquad \text{for }i\geq 0
    \;.
$$
\end{prop}

\begin{proof} 
Since $\Phi$ is a Lie algebra morphism it is enough to recall that $x_{2+i}=\frac1{i!}\ad(e)^i(x_2)$ for $i\geq 0$.
\end{proof}

If we take into account Corollaries \ref{cor:extension} and  \ref{cor:levelN-standard-relations-extension}  then we get the following result.

\begin{thm}\label{thm:Witt>equivalent}
Let $V$ be an $\sl(2)$-module defined by a representation $\sigma$ and let  $T\in {\End}_{\Gamma_>}( V,\sigma)$.

The following conditions are equivalent:
	\begin{enumerate}
	\item $T$ gives rise to a $\Witt_>$-module structure on $V$ compatible with $\sigma$, that is $T\in {\End}_{\Witt_>}( V,\sigma)$, 
	\item for $i,j\geq 0$ there holds 
	{\small $$[\frac1{i!}{\ad(\sigma(e))^i}(T),\frac1{j!}{\ad(\sigma(e))^j}(T)]=(i-j)\frac1{(i+j+2)!}\ad(\sigma(e))^{i+j+2}(T)\,,
$$}

	\item there exists $N$ such that for $k\geq N$ there holds
{\small $$[T,\frac1{(2k-1)!}\ad(\sigma(e))^{2k-1}(T)]=-\frac{2k-1}{(2k+1)!}\ad(\sigma(e))^{2k+1}(T)\,.
$$}
 
\end{enumerate}
\end{thm}

\subsection{Extending to $\Witt_{<}$}

Since the Chevalley involution $\Theta$ in $\Witt$ exchanges $\Witt_>$ and $\Witt_{<}$, the results of the previous section produce automatically analogous results for $\Witt_{>}$. In this section we just give their statements.

\begin{prop}
Let $V$ be an $\sl(2)$-module defined by a representation  $\sigma$.
The space ${\Gamma}_<( V,\sigma)$ of ${\Gamma}_<$-module structures on $V$ compatible with $\sigma$ can be identified with the set ${\End}_{\Gamma_<}( V,\sigma)\subset\End(V)$ formed by those endomorphisms $S\in \End(V)$ which satisfy the equations 
	\begin{equation}\label{eq:g-L_{-2}}
	    	\begin{aligned}
		   & \ad(\sigma(e))(S) =  -3\sigma(f), \\
	    	  & \ad(\sigma(h))(S) = -4S,
		  \end{aligned}
	    \end{equation}by mapping $\phi\in {\Gamma}_<( V,\sigma)$ to $S=\phi(x_{-2})\in \End(V)  $.

Therefore the space ${\Witt}_<( V,\sigma)$ of ${\Witt}_<$-module structures on $V$ compatible with $\sigma$ can be identified with a subset $\End_{\Witt_<}(V,\sigma)\subset {\End}_{\Gamma_<}( V,\sigma) $.
\end{prop}

\begin{prop}\label{prop:SyieldsRepr}
Let $V$ be an $\sl(2)$-module defined by a representation $\sigma$.
Given  $T\in {\End}_<( V,\sigma)$ the corresponding Lie algebra morphism  $\Phi\in \Hom_{\text{\bf LieAlg}_\bullet}(\lmax{\Gamma}_<,\End_\bullet{}(V))$ satisfies
	\begin{equation}\label{eq:rho(L_i)def2}
    \Phi(x_i)= -\frac1{i+2} \ad(\sigma(f))(\Phi(x_{i+1}))
    \qquad \text{for }i<-2
    \;,
    \end{equation} and therefore 
$$\Phi(x_{-(2+i)})=\frac1{i!} \ad(\sigma(f))^i(S)
    \qquad \text{for }i\geq 0
    \;.
$$

\end{prop}

\begin{thm}\label{thm:Witt<equivalent}
Let $V$ be an $\sl(2)$-module defined by a representation $\sigma$ and let $S\in {\End}_{\Gamma_<}( V,\sigma)$. 

The following conditions are equivalent:
\begin{enumerate}
	\item $S$  gives rise to a $\Witt_<$-module structure on $V$ compatible with $\sigma$; that is $S\in {\End}_{\Witt_<}( V,\sigma)$, 
	\item for $i,j\geq 0$ there holds
$$[\frac1{i!}{\ad(\sigma(f))^i}(S),\frac1{j!}{\ad(\sigma(f))^j}(S)]=(i-j)\frac1{(i+j+2)!}\ad(\sigma(f))^{i+j+2}(S)\,,
$$ 
	\item there exists $N$ such that for $i\geq N$ there holds
$$[\frac1{(2k-1)!}{\ad(\sigma(f))^{2k-1}}(S),{\ad(\sigma(f))}(S)]=-\frac{2k-1}{(2k+1)!}\ad(\sigma(f))^{2k+1}(S)\,.
$$ 
\end{enumerate}
\end{thm}

\subsection{Extending to $\Witt$. Compatibility of extensions}

The extendability results for $\Witt{}$ can be formulated in terms of compatibility of extensions to $\Witt_{>}$ and $\Witt_{<}$, thus in this section we limit ourselves to give their statements without proofs.

\begin{prop}\label{prop:LieAlgWitt}
Let $V$ be an $\sl(2)$-module defined by a representation  $\sigma$.
The space ${\Gamma}( V,\sigma)$ of ${\Gamma}$-module structures on $V$ compatible with $\sigma$ can be identified with the set ${\End}_\Gamma( V,\sigma) $ formed by those pairs $(S,T)$ of endomorphisms $S, T\in \End(V)$ which satisfy the equations \begin{equation}\label{eq:g-L_{-2}-L_2}
	    	\begin{aligned}
		    \ad(\sigma(e))(S) &=  -3\sigma(f), \\
	    	   \ad(\sigma(h))(S) &= -4S,\\
		   [S,T] &=2\sigma(h),\\
		   \ad(\sigma(f))(T) &=  3\sigma(e), \\
	    	   \ad(\sigma(h))(T) &= 4T,
		  \end{aligned}
	    \end{equation}by mapping $\phi\in {\Gamma}( V,\sigma)$ to $(S=\phi(x_{-2}),T=\phi(x_2))\in \End(V) \times \End(V)  $.
Therefore, the space ${\Witt}( V,\sigma)$ of ${\Witt}$-module structures on $V$ compatible with $\sigma$ can be identified with a subset $\End_{\Witt}(V,\sigma)$ of ${\End}_{\Gamma}( V,\sigma) $.
\end{prop}

\begin{defn} Let $(V,\sigma)$ be a non trivial $\sl(2)$-module. We say that two structures  of  ${\Gamma}_<$-module $(V,\phi_<)$ and ${\Gamma}_>$-module $(V,\phi_>)$ are compatible if there exists a structure of ${\Gamma}$-module $(V,\phi)$ whose restrictions to ${\Gamma}_<$ and ${\Gamma}_>$ are the given $\phi_<$, $\phi_>$, respectively. Analogously, we say that two structures  of  ${\Witt}_<$-module $(V,\rho_<)$ and ${\Witt}_>$-module $(V,\rho_>)$ are compatible if there exists a structure of ${\Witt}$-module $(V,\rho)$ whose restrictions to ${\Witt}_<$ and ${\Witt}_>$ are the given $\rho_<$, $\rho_>$, respectively.

\end{defn}

Now we can reformulate Proposition \ref{prop:LieAlgWitt} as follows.

\begin{prop} Let $(V,\sigma)$ be a non trivial $\sl(2)$-module. Two structures  of  ${\Gamma}_<$-module $(V,\phi_<)$ and ${\Gamma}_>$-module $(V,\phi_>)$ determined by $S\in{\End}_{\Gamma_<}(V,\sigma)$, $T\in{\End}_{\Gamma_>}(V,\sigma)$, respectively, are compatible, if and only if $[S,T]=2\sigma(h)$. That is 
$$\End_\Gamma(V,\sigma)=\{(S,T)\in {\End}_{\Gamma_<}(V,\sigma)\times {\End}_{\Gamma_>}(V,\sigma):[S,T]=2\sigma(h)\}.
$$ In the same way one has \begin{align*}
\End_\Witt(V,\sigma)=\{(S,T)\in {\End}_{\Witt_<}(V,\sigma)\times {\End}_{\Witt_>}(V,\sigma)&:\\ [S,T]&=2\sigma(h)\}.
\end{align*}
 
\end{prop}

\section{Representations of the Virasoro algebra}

In this section we study the extensions of $\sl(2)$-representations to the Virasoro algebra. Mostly we just give the statement of the results and do not provide proofs since they are quite similar to those given for the Witt algebras. 

\subsection{$\Witt$ and $\Vir$}

It is well known that the Witt algebra $\Witt$ admits a unique central extension $\Vir$, called the Virasoro algebra which sits in the following exact sequence of Lie algebras

$$0\to\langle C\rangle\to \Vir\to\Witt\to 0.
$$ $\Vir$ has a basis obtained by adding $C$, as a new zero degree element,  to the graded basis 
$\{L_k\}_{k\in\Z}$ of $\Witt$  and its Lie bracket is defined by:
\begin{align*}[L_i,L_j] &=(i-j) L_{i+j}+\frac{1}{12}(i^3-i)\delta_{i+j,0}\, C,\\
[\Vir,C] &=0.
\end{align*} It follows that $\langle C\rangle$ is an ideal of $\Vir$ and coincides with its center. Moreover, it is well known that in fact this is the unique proper ideal of $\Vir$. 

Note that the restriction of the central extension to the subalgebra $\Witt_>\subset\Witt$ (resp. $\Witt_<$) is trivial and, thus,  $\Witt_>$ (resp. $\Witt_<$) will also be understood as a subalgebra of $\Vir$.

\subsection{$\Vir$ and partial Lie algebras}

In a similar way the partial Witt algebra $\Gamma$ admits a partial Lie algebra central extension $\mathcal V$ of size $(-2,2)$, called the partial Virasoro algebra. $\mathcal V$ has a basis obtained by adding a new zero degree element $y_0$, to the graded basis $\{x_{-2},x_{-1},x_0,x_1,x_2\}$ of $\Gamma$ and its lie bracket is defined by:
\begin{align*}[x_i,x_j] &=(i-j) x_{i+j}+\frac{1}{12}(i^3-i)\delta_{i+j,0}\, y_0,\quad \text{if}\  -2\leq i, j, i+j\leq 2,\\
[\mathcal V,y_0] &=0.
\end{align*} Therefore $\langle y_0\rangle$ is the center of $\mathcal V$ and there is a central extension of partial Lie algebras 
$$0\to\langle y_0\rangle\to \mathcal V\to \Gamma\to 0.
$$

\begin{prop}\label{prop:pariso-vir} There is a natural isomorphism relating the Virasoro Lie algebra and its partial counterpart:

$$\Vir \simeq \lmin{\mathcal V}.
$$
\end{prop}

\begin{proof} 
Recall that the ideals of the Lie algebra $\Vir$ are  $0$, $\langle C\rangle$ or $\Vir$. Proceeding as in the  the proof of Proposition~\ref{prop:pariso}, we conclude.
\end{proof}
%
%
%

Indeed, the previous  map corresponds to the isomorphism of partial Lie algebras $\Par_{-2}^2\Vir\overset{\sim}\to {\mathcal V}$ which sends $L_i$ to $x_i$ and $C$ to $y_0$. Arguing as in  \S\ref{subsec:WittandPartial}, we obtain  a short exact sequence of graded Lie algebras 
	$$0\to I(\Vir)\to\lmax{\mathcal V} \xrightarrow{\pi}\Vir\to 0. $$
where $\pi$ is the graded Lie algebra morphism induced by mapping $x_i\in\mathcal V $ to $L_i$ and $y_0$ to $C$.

\begin{defn} 
We say that $I(\Vir)$ is the ideal associated to the graded Lie algebra $\Vir$.
\end{defn}

\begin{prop}\label{prop:triang-desc-vir} One has the following triangular decomposition into graded Lie subalgebras \begin{align*}
\lmax{\mathcal V} &=\F(\langle  x_{-2},x_{-1} \rangle) \oplus \langle  x_0,y_0 \rangle \oplus \F(\langle  x_1,x_2 \rangle).\end{align*}
\end{prop}

\begin{cor}\label{cor:idealstructure-vir} The ideal $I(\Vir)$ associated to the Witt algebra has the following decomposition into graded ideals

$$ I(\Vir)=I^-(\Vir)\oplus I^+(\Vir),
$$ where $I^-(\Vir)\subset \F(\langle  x_{-(2+i)} \rangle_{i\geq 0})$,  $I^+(\Vir)\subset \F(\langle  x_{2+i} \rangle_{i\geq 0})$ and 
$$\supp(I^-(\Vir))\subset  \Z_{\leq -5},\quad \supp(I^+(\Vir))\subset  \Z_{\geq 5}.
$$
\end{cor}



\begin{defn}\label{defn:standard-relations-vir}
The ideal $R(\Vir)$ of standard relations for the Virasoro algebra is the ideal of $\lmax{\mathcal V}$ generated by \begin{align*}
r_{2+i,2+j}&:=[x_{2+i},x_{2+j}]-(i-j)x_{2+i+j+2}\in \F(\langle  x_{2+i} \rangle_{i\geq 0}),\\
r_{-2+i,-2+j}&:=[x_{-2+i},x_{-2+j}]-(i-j)x_{-2+i+j-2}\in \F(\langle  x_{-(2+i)} \rangle_{i\geq 0}),
\end{align*}
for every $i,j\leq 0$, $i<j$.
\end{defn}

\begin{prop}\label{prop:standard-relations-vir} The  Virasoro algebra satisfies 
$$R(\Vir)=I(\Vir).
$$
\end{prop}



Notice that we have an inclusion $\Sigma=\langle x_{-1},x_0,x_1\rangle\hookrightarrow \mathcal V$ of the Lie algebra $\Sigma$ isomorphic to $\sl(2)$. Therefore if we have a $\mathcal V$-module structure on $V$ defined by a representation $\phi\colon \mathcal V \to\End(V)$, then its restriction to $\Sigma\subset \mathcal V$ defines an $\sl(2)$-module structure on $V$.

\begin{defn} Given an $\sl(2)$-module $V$ defined by a representation $\sigma\colon \sl(2)\to\End(V)$ we say that a $\mathcal V$-module structure defined by a representation $\phi\colon \mathcal V \to\End(V)$ is compatible with the given $\sl(2)$-module structure if its restriction to $\Sigma\simeq \sl(2)$ is $\sigma$, that is $\phi_{|\Sigma}=\sigma$.
\end{defn}

\begin{prop}\label{prop:nontrivial-vir} If $\phi\colon \mathcal V \to\End(V)$ is a non trivial representation of the partial Virasoro algebra that is non trivial on its center, then either $\phi$ is injective or the induced $\sl(2)$-representation  $\sigma\colon \Sigma \to \End(V)$ is trivial and $\Ker\phi=\Sigma$. Therefore if $\phi$ is non trivial on the center and the induced $\sl(2)$-module structure on $V$ is non trivial then $\phi$ is necessarily injective.

\end{prop}


\begin{prop} If $\phi\colon \mathcal V \to\End(V)$ is a representation of the partial Virasoro algebra, then $\End_\bullet(V)$ is a $\Z$-graded Lie algebra which is a subalgebra of $\End(V)$. Moreover, if the $\sl(2)$-module structure induced on $V$ is non trivial then there is an injection $\phi\colon \mathcal V \hookrightarrow \Pa\!\End_\bullet(V)$ and $\supp(\End_\bullet(V))=\supp(\Vir)=\Z$.
\end{prop}

The universal property of $\lmax{\mathcal V}$ gives the following key result.

\begin{thm} Every representation $\phi\colon \mathcal V \to\End(V)$ of  the partial Virasoro algebra extends in a unique way to give a representation $\Phi\colon \lmax{\mathcal V}\to \End_\bullet(V)$.
\end{thm}

\begin{cor}\label{cor:extension-vir} A representation $\phi\colon \mathcal V \to\End(V)$ of the partial Virasoro algebra extends to a representation $\rho\colon \Vir\to \End_\bullet(V)$ of the Virasoro algebra if and only if $\Phi\colon \lmax{\Lambda}\to \End_\bullet(V)$ vanishes on the standard relations of $\Vir$. 
\end{cor}

\begin{rem} If a representation $\phi\colon \mathcal V \to\End(V)$ of the partial Virasoro algebra extends to a representation $\rho\colon \Vir\to \End_\bullet(V)$ of the Virasoro algebra then we have a commutative diagram $$\xymatrix{ \lmax{\mathcal V}\ar[dr]_\Phi \ar[rr]^\pi& & \Vir\ar[dl]^\rho\\ &\End_\bullet(V) &}$$ and therefore, since $\pi(x_i)=L_i$, $\pi(y_0)=C$, we have $$\rho(L_i)=\Phi(x_i),\quad\rho(C)=\Phi(y_0).$$
\end{rem}

\begin{thm}\label{thm:reduced-standard-relations-vir} The ideal $R(\Vir)$ of standard relations for the Virasoro algebra is generated by \begin{align*}
r_{2,2k+1} &=[x_2,x_{2k+1}]+(2k-1)x_{2k+3}, \\
r_{-(2k+1),-2} &=[x_{-(2k+1)},x_{-2}]+(2k-1)x_{-(2k+3)},\end{align*} for $k\geq 1$.
\end{thm}

\begin{cor}\label{cor:levelN-reduced-standard-relations-vir} For any integer $N\geq 1$ the ideal $R(\Witt_>)$ is generated by \begin{align*}
r_{2,2k+1} &=[x_2,x_{2k+1}]+(2k-1)x_{2k+3}, \\
r_{-(2k+1),-2} &=[x_{-(2k+1)},x_{-2}]+(2k-1)x_{-(2k+3)},\end{align*} for $k\geq N$.
\end{cor}

\begin{defn}\label{defn:reduced-standard-relations-vir}
The set of generators given in Corollary \ref{cor:levelN-reduced-standard-relations-vir} is called the level $N$ reduced standard relations for the Virasoro algebra.
\end{defn}

Taking together Corollaries \ref{cor:extension-vir} and \ref{cor:levelN-reduced-standard-relations-vir} gives the following criterion for extending representations of the partial Virasoro algebra.

\begin{cor}\label{cor:levelN-standard-relations-extension-vir} A representation $\phi\colon \mathcal V \to\End(V)$ of the partial Virasoro algebra extends to a representation $\rho\colon \Vir\to \End_\bullet(V)$ of the Virasoro algebra if and only if $\Phi\colon \lmax{\mathcal V}\to \End_\bullet(V)$ vanishes on the level $N\geq 1$ reduced standard relations of $\Vir$. 
\end{cor}

\begin{rem}\label{rem:clave-vir}
Thus, the space of $\Vir$-module structures on a vector space $V$, i.e $\Hom_{\text{\bf LieAlg}_\bullet}(\Vir,\End(V))$, is a subset of the space of $\mathcal V$-module structures on $V$, i.e $\homparlie(\mathcal V,\End(V))$. In particular the subset $\Hom_{\text{\bf LieAlg}_\bullet}^\sigma(\Vir,\End_\bullet(V))$ of $\Hom_{\text{\bf LieAlg}_\bullet}(\Vir,\End(V))$ defined by those $\Vir$-module structures on $V$ compatible with a given $\sl(2)$-module structure $\sigma$ on $V$ (i.e $\Hom_{\text{\bf LieAlg}_\bullet}^\sigma(\Vir,\End_\bullet(V))=(\iota^*)^{-1}(\sigma)$) can be identified with a subset of the space $\homparlie^\sigma(\mathcal V,\End(V))$ formed by those $\mathcal V$-module structures that are compatible with $\sigma$.
\end{rem}

\begin{thm} For any integer $k\geq 2$ the ideal $\langle r_{-(k-1)},\ldots,r_{(k-1)}\rangle $ of $\lmax{\mathcal V}$ does not contain neither $r_{-k}$, nor $r_k$.
\end{thm}

\begin{rem} Therefore if a representation $\phi\colon \mathcal V \to\End(V)$ of the partial Virasoro algebra is such that $\Phi\colon \lmax{\mathcal V}\to \End_\bullet(V)$ vanishes on the first $(k-1)$ standard relations then, in general this does not imply that $\phi$ will extend to a representation $\rho\colon \Vir\to \End_\bullet(V)$ of the Virasoro algebra.\end{rem}

\subsection{Extending to $\Vir$}

We have seen in Remark \ref{rem:clave-vir} that given an $\sl(2)$-module structure on $V$ defined by $\sigma\in \Hom_{\text{\bf LieAlg}_\bullet}(\sl(2),\End(V))$, the space of $\Vir$-module structures on $V$ compatible with $\sigma$ can be identified with a subset of the space $\homparlie^\sigma(\mathcal V,\End_\bullet(V))$ of $\mathcal V$-module structures on $V$ compatible with $\sigma$.

\begin{prop}\label{prop:LieAlgVirasoro}
Let $V$ be an $\sl(2)$-module defined by a representation  $\sigma$.
The space ${\mathcal V}( V,\sigma)$ of ${\mathcal V}$-module structures on $V$ compatible with $\sigma$ can be identified with the set ${\End}_{\mathcal V}( V,\sigma) $ formed by those triples $(S,K,T)$ of endomorphisms $S,K, T\in \End(V)$ which satisfy the equations \begin{equation}\label{eq:g-L_{-2}-C-L_2}
	    	\begin{aligned}
		    \ad(\sigma(e))(S) &=  -3\sigma(f), \\
	    	   \ad(\sigma(h))(S) &= -4S,\\
		   [K,S] &=0,\\
		   [S,T] &=2\sigma(h)-\frac{1}{2} K,\\
		    [K,T] &=0,\\
		   \ad(\sigma(f))(T) &=  3\sigma(e), \\
	    	   \ad(\sigma(h))(T) &= 4T,
		  \end{aligned}
	    \end{equation}by mapping $\phi\in {\mathcal V}( V,\sigma)$ to $(S=\phi(x_{-2}),K=\phi(y_0),T=\phi(x_2))\in \End(V) \times \End(V)\times \End(V)  $.
Therefore, the space ${\Vir}( V,\sigma)$ of ${\Vir}$-module structures on $V$ compatible with $\sigma$ can be identified with a subset $\End_{\Vir}(V,\sigma)$ of ${\End}_{\mathcal V}( V,\sigma) $.
\end{prop}

\begin{defn} Let $(V,\sigma)$ be a non trivial $\sl(2)$-module. We say that two structures  of  ${\Gamma}_<$-module $(V,\phi_<)$ and ${\Gamma}_>$-module $(V,\phi_>)$ are Virasoro compatible if there exists a structure of ${\mathcal V}$-module $(V,\phi)$ whose restrictions to ${\Gamma}_<$ and ${\Gamma}_>$ are the given $\phi_<$, $\phi_>$, respectively. Analogously, we say that two structures  of  ${\Witt}_<$-module $(V,\rho_<)$ and ${\Witt}_>$-module $(V,\rho_>)$ are Virasoro compatible if there exists a structure of ${\Vir}$-module $(V,\rho)$ whose restrictions to ${\Witt}_<$ and ${\Witt}_>$ are the given $\rho_<$, $\rho_>$, respectively.

\end{defn}

Now we can reformulate Proposition \ref{prop:LieAlgVirasoro} as follows.

\begin{thm}\label{thm:LieAlgVirasoro2}
 Let $(V,\sigma)$ be a non trivial $\sl(2)$-module. Two structures  of  ${\Gamma}_<$-module $(V,\phi_<)$ and ${\Gamma}_>$-module $(V,\phi_>)$ determined by $S\in{\End}_{\Gamma_<}(V,\sigma)$, $T\in{\End}_{\Gamma_>}(V,\sigma)$, respectively, are Virasoro compatible, if and only if $K=4\sigma(h)-2[S,T]$ commutes with $S$ and $T$. That is \begin{align*}
\End_{\mathcal V}(V,\sigma)=\{(S,T)\in {\End}_{\Gamma_<}(V,\sigma)\times {\End}_{\Gamma_>}(V,\sigma):[K,S] &=0\\ [K,T]&=0\}.
\end{align*}
 In the same way one has \begin{align*}
\End_{\Vir}(V,\sigma)=\{(S,T)\in {\End}_{\Witt_<}(V,\sigma)\times {\End}_{\Witt_>}(V,\sigma):[K,S] &=0\\ [K,T]&=0\}.\end{align*}
\end{thm}

\begin{rem} Notice that if $S, T$ are such that  $K=\lambda\cdot {\operatorname{Id}}_V$ for a certain constant $\lambda$, then it is automatic that the ${\Gamma}_<$  and ${\Gamma}_>$ module structures determined by them are Virasoro compatible.
\end{rem}

\section{Weight Modules}

In this section we study the question of finding structures of $\Witt_{>}$-modules compatible with a given structure of $\sl(2)$ weight module. We remark in passing that the extended module if it exists then it continues to be a weight module for $\Witt_{>}$ and in the same way Harish-Chandra modules and those in the BGG category are preserved as well.  We also consider the same problem for the Lie algebras  $\Witt_{<} ,\Witt$ and $\Vir$. We will find positive and explicit results for all the simple infinite dimensional modules and the class of dense modules. We also exhibit modules whose $\sl(2)$-module structure can not be extended to  $\Witt_{>}$. Our main reference for the theory of weight $\sl(2)$-modules will be the book by Mazorchuk \cite[Chapter 3]{Mazor}, see also the references therein. We will also use his notation.

\subsection{Decomposing the Category}\label{subset:decompCategory}

Let $\bar\Mf$ be the full subcategory of the category of $\sl(2)$-modules whose objects are weight modules with finite-dimensional weight spaces. In order to simplify our problem, let us see how this category can be decomposed (\cite[\S~3.6]{Mazor}). For $\xi\in\C/2\Z$, which can be thought of as a subset $\xi\subset \C$, we consider the full subcategory whose objects are given by
	$$
	\bar\Mf^{\xi}\,:=\, \{ M\in \bar\Mf \text{ s.t. } \supp(M) \subseteq \xi\}  \, .
	$$
Further, for each $\tau\in\C$, we consider the full subcategory of $\bar\Mf^{\xi}$ 
	$$
	\bar\Mf^{\xi,\tau}\,:=\, \{ M\in \bar\Mf^{\xi}  \text{ s.t. $(c-\tau)\vert_{M_{\lambda}}$ is nilpotent for every $\lambda\in\xi$} \}
	\, ,
	$$
where $c$ is the Casimir operator and $M_{\lambda}$ denotes the weight space corresponding to $\lambda$; that is, $M_{\lambda}:=\{m\in M\vert h(m)=\lambda m\}$.

Observe that  any $M\in \bar\Mf$  decomposes as
	$$
	M \, = \, \bigoplus_{\xi\in\C/2\Z} \Big(  \bigoplus_{\tau\in\C}  
	M_{\xi}(\tau)
	\Big)
	$$
where
	$$
	\begin{gathered}
	M_{\xi}(\tau)\,:= \, \oplus_{\lambda\in \xi} M_{\lambda}(\tau)
	\\
	M_{\lambda}(\tau)\, :=\, \{ m\in M_{\lambda} \text{ s.t. $(c-\tau)^k m=0$ for $k >>0$}\} \, .
	\end{gathered}
	$$
Note that $M_{\xi}(\tau)$ belongs to $\bar\Mf^{\xi,\tau}$ and that $\oplus_{\tau\in\C} M_{\xi}(\tau)$ lies in $\bar\Mf^{\xi}$.

Furthermore, given $M_i\in \bar\Mf^{\xi_i,\tau_i}$ for $i=1,2$, it holds that
	$$
	\Hom_{\bar\Mf}(M_1, M_2)\, = \, 0 \qquad \text{ unless $\xi_1=\xi_2 $ and $\tau_1=\tau_2$}
	\, .
	$$

Accordingly, finding a $\Witt_>$-module structure on $M\in \bar\Mf $ is equivalent to finding such structure on each submodule $M_{\xi}(\tau) \in \bar\Mf^{\xi,\tau}$ or, what is tantamount, our problem can be reduced to study the case of modules in the category $\bar\Mf^{\xi,\tau}$. 

Recall that $\bar\Mf^{\xi,\tau}$ is an abelian, Krull-Schmidt category and that every object has finite length. Let us describe its simple objects (\cite[Theorem~3.32 and Proposition~3.55]{Mazor}).

For $n\in{\mathbb N}$ there is a unique simple $\sl(2)$-module of dimension $n$ (up to isomorphism). It can be explicitly constructed as follows. Let ${\mathbf V}^{(n)}$ be the vector space generated by $\{v_0, \ldots, v_{n-1}\}$ and consider the action given by
	{\small $$
	\begin{aligned}
	f(v_i) \,:&=\,\begin{cases} v_{i+1} & \text{ for } i\neq n-1 \\ 0 & \text{ for } i=n-1 \end{cases}
	\\
	h( v_i)\,:&=\, (n-1-2i) v_i
	\\
	e(v_i) \,:&=\,\begin{cases} i(n-i) v_{i-1} & \text{ for } i\neq 0 \\ 0 & \text{ for } i=0 \end{cases}
	\end{aligned}
	$$}
Observe that in this case, $\supp {\mathbf V}^{(n)}=\{1-n, 3-n,\ldots, n-3,n-1\}$. The Casimir operator acts by multiplication by $n^2$.

The second example consists of Verma modules which are particular instances of highest weight modules. The Verma module of highest weight $\lambda\in\C$, $M(\lambda)$, is the vector space generated by $\{v_i\vert i=0,1,2,\ldots\}$ with the $\sl(2)$-action
	\begin{equation}\label{eq:actionVerma}
	\begin{aligned}
	f(w_{i}) \,& =\, w_{i+1}
	\\
	h(w_{i}) \,&=\, (\lambda-2 i) w_{i}
	\\
	e(w_{i}) \,&=\, \begin{cases} i (\lambda -i+1) w_{i-1} & \text{ for } i\neq 0 \\ 0, & \text{ for } i=0\end{cases}
	\end{aligned}
	\end{equation}
It holds that $\supp M(\lambda)=\{\lambda -2i\vert  i=0,1,2,\ldots\}$ and that the Casimir acts as the scalar $(\lambda+1)^2$. Further, $M(\lambda)$ is simple if and only if $\lambda\neq 0,1,2,\ldots$.

The third case is concerned with lowest weight modules. For $\lambda \in \C$, consider the vector space $\bar M(\lambda)$ freely generated by $\{w_i \vert i=0,1,2,\ldots\}$ and endow it with  the following $\sl(2)$-action:
	\begin{equation}\label{eq:actionbarM}
	\begin{aligned}
	e(w_{i}) \,& =\, w_{i+1}
	\\
	h(w_{i}) \,&=\, (\lambda+2i) w_{i}
	\\
	f(w_{i}) \,&=\, \begin{cases} -i (\lambda +i-1) w_{i-1} & \text{ for } i\neq 0 \\ 0, &  \text{ for }i=0\end{cases}
	\end{aligned}
	\end{equation}
hence, $\supp \bar M(\lambda)=\{\lambda + 2i\vert  i=0,1,2,\ldots\}$ and $c$ acts as the scalar $(\lambda -1)^2$. Finally, $\bar M(\lambda)$ is simple for $\lambda\neq 0,-1,-2,\ldots$.

The last example is the case of dense modules. For $\xi\in \C/2\Z$ and $\tau\in \C$, the dense module ${\mathbb V}(\xi,\tau)$ is the vector space generated $\{v_{\mu} \vert \mu \in \xi\}$ carrying the following $\sl(2)$-action
	\begin{equation}\label{eq:dense}
	\begin{aligned}
	f(v_{\mu}) \,& =\, v_{\mu-2}
	\\
	h(v_{\mu}) \,&=\, \mu v_{\mu}
	\\
	e(v_{\mu}) \,&=\, \frac14 (\tau-(\mu+1)^2) v_{\mu+2}
	\end{aligned}
	\end{equation}
It is straightforward to see that $\supp {\mathbb V}(\xi,\tau)= \xi$ and that $c$ acts by $\tau$. Furthermore, it is simple if and only if $\tau\neq (\mu+1)^2$ for all $\mu\in\xi$.

The four instances given above  exhaust all simple weight modules. More precisely, the simple objects of $\bar\Mf^{\xi,\tau}$ are given by the following (\cite[Theorem~3.55]{Mazor})

\begin{thm}\label{thm:simplemodinMxitau}
\begin{enumerate}
	\item if $\tau\neq (\mu+1)^2$ for all $\mu\in\xi$, then $\bar\Mf^{\xi,\tau}$ has only one simple object, namely ${\mathbb V}(\xi,\tau)$;
	\item if $\tau= (\mu+1)^2$ for exactly one $\mu\in\xi$, then $\bar\Mf^{\xi,\tau}$ has two simple objects, namely $M(\mu)$ and $\bar M(\mu+2)$;
	\item if $\tau= (\mu+1)^2=(\mu+2n+1)^2$ for some $n\in{\mathbb N}$, then $\mu=-n-1$, $\tau=n^2$ and   $\bar\Mf^{\xi,\tau}$ has three simple objects, namely $M(-n-1), {\mathbf V}^{(n)}$ and $\bar M(n+1)$.
\end{enumerate}
\end{thm}

In the following subsections we study the possible extensions of the $\sl(2)$-module structure for the three cases of this Theorem. 

\subsection{First Case}\label{subsec:dense}

The category  $\bar\Mf^{\xi,\tau}$ will be studied for $\tau\neq (\mu+1)^2$ for all $\mu\in\xi$. We begin with the case of the simple object on this category; namely, we wonder whether the dense module ${\mathbb V}(\xi,\tau)$ carries a compatible structure of $\Witt_>$-module. After a detailed study of this situation, we will sketch the cases of $\Witt_<$, $\Witt$ and $\Vir$. Recall that the Pochhammer symbol, $P(z,n)$, is defined as
	$$
	P(z,n) \,=\, \,\begin{cases}
		z(z+1)\dots (z+n-1)  & \text{ for } n\geq 0
		\\
		\frac1{(z+n)(z+n+1)\dots (z-1)} & \text{ for } n\leq -1
		\end{cases}
	$$
where $z\in\C$ and $n\in \Z$. Observe that it makes sense for $z\in\operatorname{Mat}_{l\times l}({\mathbb C})$ and $n\geq 0$ while for $n<0$ one should impose that negative integer numbers are not  eigenvalues of $z$. Finally, for $z\in \C$, it holds that $P(z,n)=\frac{\Gamma(z+n)}{\Gamma(z)}$  where $\Gamma$ is the Gamma function.

\subsubsection{Extending to $\Witt_>$}

We prove now a result that is valid for arbitrary dense modules. 

\begin{prop}\label{prop:denseW>}
Let $\xi, \tau$ be arbitrary. The dense $\sl(2)$-module ${\mathbb V}(\xi,\tau)$ admits at most two compatible structures of $\Witt_>$-module. Moreover, these structures are explicitly given by:
	{\small $$
	\rho_{>}^{\pm}(L_i)v_{\mu} \,:=\,  \Big(
	 \frac{(-1)^i}{2}
	\big( i(\pm \sqrt{\tau}-1)-\mu \big)
	\, P\big(\frac12(1+\mu\pm \sqrt{\tau}), i\big)\Big) v_{\mu+2i}
	$$}where $P(z,n)$ is the Pochhammer symbol, $i\geq -1$ and $\mu\in\xi$. Moreover, both extensions coincide if and only if $\tau=0$ or $\tau=1$.
\end{prop}

\begin{proof}
The idea relies on Corollary~\ref{thm:Witt>equivalent}.

\emph{Step 1}. There exists $T\in \End({\mathbb V}(\xi,\tau))$ fulfilling equations~(\ref{eq:g-L_2}). Indeed, applying the second identity of~(\ref{eq:g-L_2}) to $v_{\mu}$ we obtain:
	{\small $$
	4T (v_{\mu}) \,=\,   [h, T] v_{\mu} \,=\,
	h(T ( v_{\mu})) \, - \, T( h ( v_{\mu}))
	 \,=\,
	h(T ( v_{\mu})) \, - \, \mu T( v_{\mu})
	\, ,$$  }
which shows that, if $T$ does exist, it must be of the form:
	$$
	T (v_{\mu}) \,=\, a_{\mu}^2 v_{\mu+4}
	$$
for certain complex numbers $a_{\mu}^2\in\C$. Let us use the first identity of~(\ref{eq:g-L_2}) to deduce an explicit expression for $a_{\mu}^2$ and, consequently, for $T$. Applying that relation to $v_{\mu}$ we get:
	$$
	3e(v_{\mu}) \,=\, [f, T] (v_{\mu})
	\,=\, f(a_{\mu}^2 v_{\mu+4}) \,-\, T(v_{\mu-2})
	\, ,
	$$
and, since $V_{\mu}=<v_{\mu}>$ for all $\mu$,
	$$
	\frac34(\tau-(\mu+1)^2) \, =\, a_{\mu}^2 \, -\, a_{\mu-2}^2 \qquad \forall \mu\in\xi
	\, .
	$$

The general solution of this difference equation is:
	$$
	a_{\mu}^2
	\,=\,
	\alpha - \frac18 \mu(11+6\mu+\mu^2 -3\tau)
	$$
for a constant $\alpha\in\C$.

\emph{Step 2}. Compute $\alpha$. For $T$ as in Step 1, we consider $\rho$ as defined by equations~(\ref{eq:defrho-1012}) and~(\ref{eq:rho(L_i)def}). Demanding that the second item of Corollary~\ref{thm:Witt>equivalent} holds for $i=0$, $j=1$, we obtain that there are two possible values for $\alpha$:
	$$
	\alpha_{\pm}\,=\,
	\frac14(\tau-1)(3\pm \sqrt{\tau})
	\, .
	$$
Observe that the two possibilities coincide if and only if $\tau=0$ or $\tau=1$.

\emph{Step 3}. Compute $\rho_{>}^{\pm}(L_i)$ for a given $\alpha$. For the sake of notation, we simply write $\rho(L_i)$. Recalling that  $\rho(L_i)\in\End_{-2i}({\mathbb V}(\xi,\tau))$ for all $i\geq -1$, we may introduce constants $a_{\mu}^i$ by the defining relations:
	$$
	\rho(L_i)(v_{\mu})\,=\, a_{\mu}^i v_{\mu+2i} \qquad \text{for }\mu\in\xi \, , \, i\geq -1
	\, .
	$$
Hence, by equations~(\ref{eq:defrho-1012}) and the previous Steps, it follows that:
	$$
	\begin{aligned}
	a_{\mu}^{-1} \, & =\, 1
	\\
	a_{\mu}^0 \, & =\, -\frac12 \mu
	\\
	a_{\mu}^1 \, & =\, -\frac14(\tau-(\mu+1)^2)
	\\
	a_{\mu}^2 \, & =\, \frac14(\tau-1)(3\pm \sqrt{\tau}) - \frac18 \mu(11+6\mu+\mu^2 -3\tau)
	\end{aligned}
	$$
Writing the equation~(\ref{eq:rho(L_i)def}) in terms of the $a$'s we get:
	{\small $$
	a_{\mu}^i \,=\,
	- \frac1{i-2}\big(
	a_{\mu+2(i-1)}^1 a_{\mu}^{i-1}\, - \, a_{\mu+2}^{i-1} a_{\mu}^1\big)
	\qquad\text{ for }\mu\in\xi \, , \,i\geq 3
	\, .
	$$}
	
One can find the general explicit expression for $a_{\mu}^i $ (where $\mu\in\xi$ and $i\geq -1$), and it  reads as follows:
	{\small \begin{equation}\label{eq:amui-denso}
	a_{\mu}^i \,=\,
	 \frac{(-1)^i}{2}
	\big( i(\pm \sqrt{\tau}-1)-\mu \big)
	\, P\big(\frac12(1+\mu\pm \sqrt{\tau}), i\big)
	\, ,
	\end{equation}}
where $P(z,n)$ is the Pochhammer symbol.
	
\emph{Step 4}.  Check that the second item of  Corollary~\ref{thm:Witt>equivalent} is fulfilled. This identity can be written down in terms of $a$'s and, therefore, it suffices to show that:
	\begin{equation}\label{eq:aiaj-ajai}
    a_{\mu+2j}^{i}a_{\mu}^{j} \,-\, a_{\mu+2i}^{j} a_{\mu}^{i} \,=\, (i-j) a_{\mu}^{i+j}
    \, ,
    \end{equation}
holds true for $i,j\geq -1$ and $\mu\in\xi$. This follows from the explicit expressions obtained in Step 3.
\end{proof}

Once we know that the $\sl(2)$-module structure of any dense module ${\mathbb V}(\xi,\tau)$ can be extended, we aim at classifying it as a $\Witt_{>}$-module. One checks that ${\mathbb V}(\xi,\tau)\in\bar\Mf^{\xi,\tau}$ (for $\tau\neq (\mu+1)^2$ for all $\mu\in\xi$)  is simple and a  Harish-Chandra $\Witt_{>}$-module. Recall that  Harish-Chandra $\Witt_{>}$-modules has been classified by Mathieu (\cite{Mathieu}). Indeed, he showed that Harish-Chandra modules over the $\Witt_{>}$ algebra are exactly the highest weight modules,  the lowest weight modules and the intermediate series modules (see \cite[Theorem~2]{Mathieu}). Let us recall the construction of the latter. Consider ${\mathbf V}_{a,b}$ the $\C$-vector space freely generated by $\{w_n\vert n\in \Z\}$ endowed with the action
	\begin{equation}\label{eq:intermediate}
	\rho(L_i) w_n \, :=\, (ai+b-n)w_{n+i} \qquad i\geq -1
	\end{equation}
where $ \{L_i\}_{i\geq -1 }$ is the standard basis of $\Witt_{>}$. The nontrivial simple subquotients of modules ${\mathbf V}_{a,b}$ are called the intermediate series modules.  It holds that ${\mathbf V}_{a,b}$ is simple for $b\notin \Z$ or $a\neq 0,1$. Observe that there is an isomorphism of $\Vir$-modules ${\mathbf V}_{a,b} \simeq {\mathbf V}_{a,b+k}$ for all $k\in\Z$ (mapping $w_n$ to $w_{n+k}$). Henceforth, the isomorphism class of ${\mathbf V}_{a,b}$ depends on the class of $b$ in $\C/\Z$ rather than on $b$ itself. Let $\overline{2b}$ denote the class of $2b$  in $\C/2\Z$.

\begin{prop}\label{prop:DenseasWitt-Intermediate}
For $2a \notin -\overline{2b}$, ${\mathbf V}_{a,b}$ considered as an $\sl(2)$-module (via the embedding (\ref{eq:emb-sl2-W+})) is isomorphic to ${\mathbb V}(-\overline{2b} , (1+2a)^2)$.

Conversely, $({\mathbb V}(\xi,\tau),\rho^{\pm}_{>})$ is isomorphic to ${\mathbf V}_{a,b}$ with $-\overline{2b}=\xi$ and $a=\frac{1}{2}(-1\pm \sqrt{\tau})$ as  $\Witt_{>}$-modules.
\end{prop}

\begin{proof}
For proving the statement, we must compare the actions (\ref{eq:dense}) and (\ref{eq:intermediate}). 

For the first item, let ${\mathbf V}_{a,b}$ be given and set $\xi =-\overline{2b} , \tau= (1+2a)^2$. Note that $2a \notin -\overline{2b}$ implies that $\tau\neq (\mu+1)^2$ for all $\mu\in \xi$. Let us consider the map
	$$
	\begin{aligned}
	{\mathbf V}_{a,b} &\, \longrightarrow \,  {\mathbb V}(\xi,\tau)
	\\
	\omega_n & \,\longmapsto \, P(-a+b-n,n) v_{2(a-b+n)}\qquad n\in{\mathbb Z}
	\end{aligned}
	$$
It is easy to check that it gives rise to an isomorphism of $\sl(2)$-modules. The second claim follows from a straightforward computation. 
\end{proof}


\begin{prop}\label{prop:denseW>2} Let us assume that $\tau\neq (\mu+1)^2$ for all $\mu\in\xi$.
For every object of $\bar\Mf^{\xi,\tau}$  and any square root of its Casimir operator $c$, there exists a compatible structure of $\Witt_{>}$-module such that $c$ is central, i.e. $c$ commutes with the representation of $L_i$ for every $i\geq -1$.
\end{prop}

\begin{proof}
Let $M$ be an object of $\bar\Mf^{\xi,\tau}$. Since $M$ has finite length and ${\mathbb V}(\xi,\tau)$ is the unique simple object, let us choose a composition series 
	$$
	M^0= 0 \, \subsetneq \, M^1\, \subsetneq \, \ldots\, \subsetneq \, M^l =M
	$$
with $M^i/M^{i-1}\simeq {\mathbb V}(\xi,\tau)$. 

Bearing in mind that lowest weight modules do not belong to $\bar\Mf^{\xi,\tau}$, one has that $f$ acts injectively; that is, $f_{\mu}:M_{\mu}\to M_{\mu-2}$ is injective for all $\mu\in\xi$. Furthermore, since $\dim M_{\mu}<\infty$ does not depend on $\mu$, $f$ acts bijectively. Thus, we may fix a basis on every $M_\mu$ such that $f_{\mu}:M_{\mu}\to M_{\mu-2}$ acts by the $l\times l$ identity matrix. Second, since $M$ is a weight module, $h_{\mu}\vert_{M_{\mu}}$ is the homothety of ratio $\mu$ and, finally 
	$$
	e_\mu \, = \, \frac14(\tau-(\mu+1)^2) \operatorname{Id}_l + N_{\mu},
	$$
where $N_{\mu}$ is a strictly upper triangular  $l\times l$ matrix (since the action is compatible with the composition series). The condition $[e,f]=h$ implies that $N_{\mu}$ does not depend on $\mu$ and we denote it by $N$.

Now, following the ideas used in the first step of the proof of Proposition~\ref{prop:denseW>} we show that there exists $T\in \End(M)$ fulfilling equations~(\ref{eq:g-L_2}). The second of those equations is satisfied if and only if  $T$ maps $M_\mu$ to $M_{\mu+4}$ and therefore in the basis fixed above $T\colon M_\mu\to M_{\mu+4}$ is represented by an  $l\times l$ matrix $A_\mu$. Now the first equation in~(\ref{eq:g-L_2}) is equivalent to $$A_\mu-A_{\mu-2}=\frac34(\tau-(\mu+1)^2)\operatorname{Id}_l+3 N.$$ The general solution of this difference equation is $$A_\mu=A- \frac18 \mu(11+6\mu+\mu^2 -3\tau)\operatorname{Id}_l+\frac32 \mu N,$$ where $A$ in an $l\times l$ matrix. The Casimir operator $c$ of the $\sl(2)$-module $M$ maps $M_\mu$ to itself and a straightforward computation shows that in the basis we are using $c\colon M_\mu\to M_{\mu}$ acts by the  matrix $c=\tau \operatorname{Id}_l+4N$. Therefore, the condition that $T$ commutes with the Casimir $c$ is equivalent to $$[A,N]=0.$$

In order to be able to extend the $\sl(2)$-representation, $T$ has to satisfy  Corollary~\ref{thm:Witt>equivalent}. In particular, a straightforward but lengthy computation shows that the second item of that corollary holds for $i=0$, $j=1$, if and only if $$16 A^2 - 24  (c-1) A - (c-9) (c-1)^2=0.$$ Therefore we get $$A=\frac14(c-1)(3+ \sqrt{c}),$$ where $\sqrt{c}$ denotes any square root of the matrix $c=\tau \operatorname{Id}_l+4N$.

Now, inspired by formulas (\ref{eq:amui-denso}), we introduce matrices
	$$
	A_{\mu}^i 
	\,:=\, 
	\frac{(-1)^i}2 (i ( 1+ \sqrt{c})  -(\mu+2i)) P(\frac12(( 1 + \sqrt{c}) +\mu),i) 
	\qquad i\geq -1
	$$
where $c$ is the Casimir and scalars are thought as multiples of the identity matrix. 
 
In order to check that 
	\begin{equation}\label{eq:rhoAimu}
	\rho(L_i)\vert_{M_{\mu}}:= A_{\mu}^i
	\,:\, M_{\mu} \,\longrightarrow\, M_{\mu+2i}
	\end{equation}
(w.r.t. the previously chosen basis) defines a $\Witt_{>}$-module structure on $M$, it suffices to show that 
	$$
	A_{\mu+2j}^i A^j_{\mu} - A^j_{\mu+2i}A^i_{\mu} \,=\, (i-j) A^{i+j}_{\mu} 
	$$
for $i,j\geq -1$ and $\mu\in\xi$. But this follows easily from the definition of the $A_\mu^i$'s and the properties of the Pochhammer symbol.
\end{proof}

\begin{rem}  In particular, if $\tau\neq0$, then the Casimir operator $c$ of every $\sl(2)$-module of the category $\bar\Mf^{\xi,\tau}$, with  $\tau\neq (\mu+1)^2$ for all $\mu\in\xi$, always admits square roots $\sqrt{c}$. However, if $\tau=0$ then $\sqrt{c}$ may not exist. In both cases, there may exist infinitely many square roots of $c$. 
\end{rem}

\begin{thm}\label{thm:categorydenseW>}
Consider the category $\bar\Mf^{\xi,\tau}$  for $\tau\neq (\mu+1)^2$ for all $\mu\in\xi$ and $\tau\neq 0$. 

A square root $\sqrt{\tau}\in{\mathbb C}$ determines a fully faithful functor
	$$
	F_{\sqrt{\tau}}: \bar\Mf^{\xi,\tau} 
	\,\to \, 
	 {\Witt_{>}\text{--\,mod}}	
	$$
such that $G\circ F_{\sqrt{\tau}}=\operatorname{Id}_{\bar\Mf^{\xi,\tau} }$ where $G:{\Witt_{>}\text{--\,mod}}   \to {\sl(2)\text{--\,mod}}  $ is the forgetful functor.
\end{thm}

\begin{proof}
We continue with the notation and the ideas of the proof of Proposition~\ref{prop:denseW>2}. Let $s(x)$ be the series expansion of the function $\sqrt{x}$ at $x=\tau$ such that $s(\tau)$ is the previously chosen root $\sqrt{\tau}\in{\mathbb C}$. In fact, it holds that
	$$
	s(x)\,=\, \sqrt{\tau}\Big( 
		\sum_{n=0}^{\infty} \frac{(-1)^n (2n-3)!!}{2^n \tau^n n!} (x-\tau)^n 
		\Big)
	$$

Given an arbitrary module $M\in\operatorname{Ob}(\bar\Mf^{\xi,\tau}  )$ the series $s(c)$, where $c$ is the Casimir, is actually a polynomial in $c$ since we know that $c-\tau$ is nilpotent and $M$ has finite length. 
Thus, we define $F_{\sqrt{\tau}}(M)$ to be $M$ (as a vector space) endowed with the action, $\rho_M$, of $\Witt_{>}$ defined through equation~(\ref{eq:rhoAimu}) where we replace $\sqrt{c}$ by $s(c)$. 

Let us study the case of morphisms. Let  $M,N\in\operatorname{Ob}(\bar\Mf^{\xi,\tau}  )$ and 
	$$
	\Phi\,\in\, \Hom_{\bar\Mf^{\xi,\tau} }(M,N) \,=\,  \Hom_{\sl(2)-mod }(M,N) 
	$$
We will show that $\Phi$ is a homomorphism of $\Witt_{>}$-modules. 

Having in mind that $\Phi$ commutes with $h$, one obtains that $\Phi(M_{\mu})\subseteq N_{\mu}$ for all $\mu\in\xi$. Considering basis in $M_{\mu}$ and $N_{\mu}$ as in the proof of  the previous Proposition, we associate to $\Phi\vert_{M_{\mu}}$ a matrix $\phi_{\mu}$. Since $\Phi$ commutes with the action of $f$ and the matrix associated to $f_{\mu}:M_{\mu}\to M_{\mu-2}$ (resp. $f_{\mu}:N_{\mu}\to N_{\mu-2}$) is the identity matrix, it follows that $\phi_{\mu}$ does not depend on $\mu$; let us denote it by $\phi$. 

Similarly, the Casimir $c:M\to M$ (resp. $c:N\to N$) commutes with $h$ and $f$, hence the matrix associated to $c\vert_{M_{\mu}}$ (resp. $c\vert_{N_{\mu}}$) does not depend on $\mu$ and we denote it $c_M$ (resp. $c_N$). 

The identity of maps $c\circ \Phi=\Phi\circ c$ implies the identity of matrices $c_N \cdot \phi = \phi \cdot c_M$.  It thus follows that $A_{N,\mu}^i \cdot \phi = \phi \cdot A_{M,\mu}^i$ where $A_{M,\mu}^i $ (resp. $A_{N,\mu}^i $) denotes the matrix $A_{\mu}^i$ where $\sqrt{c}$ has been replaced by $s(c_M)$ (resp. $c_N$). Notice that this relation is equivalent to
	$$
	\rho_N (L^i) \circ \Phi \,=\, \Phi\circ \rho_M(L^i)
	$$
as claimed. We define $F_{\sqrt{\tau}}(\Phi)=\Phi\in \Hom_{\Witt_{>}-mod}(	F_{\sqrt{\tau}}(M), F_{\sqrt{\tau}}(N))$. It is straightforward to see that $F_{\sqrt{\tau}}$ is fully faithful. 
\end{proof}

There are cases in which the infinite set of relations of Corollary~\ref{cor:extension} collapses to a few constraints as is shown in the following result.

\begin{thm}\label{thm:CasimirR23forfirstCategory}
Let $M$ be an object of $\bar\Mf^{\xi,\tau}$,  with $\tau\neq (\mu+1)^2$ for all $\mu\in\xi$ and $\tau\neq 0$, $\sigma$ be its $\sl(2)$-module structure and $c$ its Casimir operator.
Given $T\in\End_{\sl(2)}(M)$, if the following relations hold
	$$
	\begin{aligned}
	\ad(\sigma(f))(T) \,&=\, 3\sigma(e)
	\\
	\ad(\sigma(h))(T)\,&=\, 4 T 
	\\
	\ad(\sigma(c))(T) \,&=\, 0
	\\
	[T , \ad(\sigma(e))(T) ] \,&=\, -\frac16 \ad(\sigma(e))^3(T)
	\end{aligned}
	$$
then, $T$ gives rise to a compatible  $\Witt_>$-module structure on $M$. 
\end{thm}

\begin{proof}
Arguing as in Proposition~\ref{prop:denseW>2}, we may fix basis for all $M_{\mu}$ such that $\sigma$ is expressed as follows: $\sigma(f)\vert_{M_{\mu}}=\operatorname{Id}$, $\sigma(h)\vert_{M_{\mu}}=\mu$ and 
	{\small $$
	\sigma(e)\vert_{M_{\mu}}\,=\, 
	\frac14(\tau-(\mu+1)^2)+N 
	\,=\, \frac14 (1+ \sqrt{c}  -(\mu+2))( 1 + \sqrt{c} +\mu) 
	$$}
where $\sqrt{c}$ is  a square root of the Casimir, $c=\tau+4N$ with $N$ a strictly upper triangular constant matrix. 

From $\ad(\sigma(h))(T)= 4 T$ we know that $T_{\mu}:=T\vert_{M_{\mu}}$ takes values in $M_{\mu+4}$. A general expression for a set $\{T_{\mu}\}$ fulfilling $f_{\mu+4}T_{\mu}- T_{\mu-2}f_\mu= 3e_{\mu}$ has the form
	$$
	T_{\mu}\,=\, \frac12 (2( 1+ \sqrt{c})  -(\mu+4)) P(\frac12(( 1 + \sqrt{c}) +\mu),2) + T^0,
	$$
where $T^0$ is a operator whose associated matrix w.r.t. the chosen basis does not depend on $\mu$; that is, it is represented by a square constant matrix. Now, the identity $\ad(\sigma(c))(T)=0$ implies that $[\sigma(c) , T^0] = [N, T^0]=0$.

Since $T^0$ commutes with the Casimir, a long but straightforward computation yields the following explicit expression for the fourth relation of the statement 
	{\small 
	$$
	0\,=\, 
	[T , \ad(\sigma(e))(T) ] +\frac16 \ad(\sigma(e))^3(T)
	\,=\, 
	8 T^0 \big( T^0-\frac12(1-c) \sqrt{c}\big) 
	$$}which allows us to consider the decomposition 
	$$
	M\simeq \ker(T^0)\oplus \ker (T^0-\frac12(1-c) \sqrt{c})\, .
	$$ 
Observing that $[\sigma(f), T^0]=0$, $[\sigma(h), T^0]=4T^0$ and that 
	{\small $$
	\begin{aligned}
	\,[\sigma(e), T^0]\vert_{M_{\mu}}\, & =\,
	[\frac14(\tau-(\mu+1)^2)  , T^0]  + [N  , T^0] 
	\,= \\ 
	&= \, \frac14\big( (\tau-(\mu+5)^2) - (\tau-(\mu+1)^2)\big) T^0 \,+\,  0 
	\,=\, 	 (6- 2\mu) T^0\vert_{M_{\mu}}
	\end{aligned}
	$$}one obtains that $\ker(T^0)$ and $ \ker  (T^0-\frac12(1-c) \sqrt{c})$ are $\sl(2)$-submodules of $M$. 
The problem thus may be restricted to these two cases. That is, we may assume that either $T^0=0$ or $T^0=  \frac12(1-c) \sqrt{c}$. 

For the first case, $T^0=0$, we set $\rho(L_i)\vert_{M_{\mu}}:= A_{\mu}^i:M_{\mu}\to M_{\mu+2i}$ where
	$$
	A_{\mu}^i 
	\,:=\, 
	\frac{(-1)^i}2 (i ( 1+ \sqrt{c})  -(\mu+2i)) P(\frac12(( 1 + \sqrt{c}) +\mu),i) 
	\qquad i\geq -1
	$$
which coincides with $\sigma(f), -\frac12\sigma(h), -\sigma(e)$ and $T$ for $i=-1,0,1,2$ respectively. And one checks, as in the proof of Proposition~\ref{prop:denseW>2} that it defines a compatible $\Witt_>$-module structure. 

The second case is analogous; it suffices to consider
	$$
	A_{\mu}^i 
	\,:=\, 
	\frac{(-1)^i}2 (i ( 1- \sqrt{c})  -(\mu+2i)) P(\frac12(( 1 - \sqrt{c}) +\mu),i) 
	\qquad i\geq -1
	$$
\end{proof}


%
%

\subsubsection{Extending to $\Witt_{<}$, $\Witt$ and $\Vir$}

\begin{prop}\label{prop:denseW<}
Let $\xi,\tau$ be arbitrary. The dense $\sl(2)$-module ${\mathbb V}(\xi,\tau)$ does not admit a compatible structure of $\Witt_<$-module if and only if the roots of the polynomial $ (x-1)^2-\tau$ lie in $\xi$.

If the root $1\mp \sqrt{\tau}$ does not lie in $\xi$, then ${\mathbb V}(\xi,\tau)$ is a $\Witt_<$-module via the map
	{\small $$
	\rho_{<}^{\pm}(L_i)v_{\mu} \,:=\,
 \Big(
	  \frac{(-1)^i}{2}
	\big( i(\pm \sqrt{\tau}-1)-\mu \big)
	\, P\big(\frac12(1+\mu\pm \sqrt{\tau}), i\big)
	\Big) v_{\mu+2i}
	$$}
where $P(z,n)$ is the Pochhammer symbol, $i\leq 1$ and $\mu\in\xi$.

In particular, if ${\mathbb V}(\xi,\tau)$ is simple, then it carries exactly
 two extended structures for $\tau\neq 0,1$ and one for $\tau=0,1$. 
 \end{prop}

\begin{proof}
Since the strategy is similar to the case of $W_>$ (see Proposition~\ref{prop:denseW>}), only the main steps will be sketched. 

\emph{Step 1}. There exists $S\in \End(V)$ fulfilling equations~(\ref{eq:g-L_{-2}}). From the identities of~(\ref{eq:g-L_{-2}}), one obtains that $S$ must be of the form
	{\small $$
	S (v_{\mu}) \,=\, a_{\mu}^{-2} v_{\mu-4}
	\quad\text{ where }\quad
	a_{\mu}^{-2}
	\,=\,
    \frac{-2\mu(11-6\mu+\mu^2-3\tau) + \alpha(\tau-1)^2}
        {((\mu-3)^2-\tau)((\mu-1)^2-\tau)}
        \, ,
	$$}
for a constant $\alpha\in\C$. The vanishing  of the denominator will be studied in Step 4.

\emph{Step 2}. Compute $\alpha$. When $\tau=1$, we will set $\alpha=0$ (and both possibilities coincide). Assume $\tau\neq 1$. For $S$ as in Step 1 and $\rho$ defined by Proposition~\ref{prop:SyieldsRepr}, the identity of the second item of Theorem~\ref{thm:Witt<equivalent} for $i=0$, $j=1$ yields the following values
	$$
	\alpha_{\pm}\,=\,
	\frac{12\pm 4\sqrt{\tau}}{1-\tau}
	\, .
	$$
Observe that the two possibilities coincide if and only if $\tau=0$.

\emph{Step 3}. Compute $\rho_{<}^{\pm}(L_i)$ for a given $\alpha$. If no confusion arises, we simple write $\rho$. Recalling that  $\rho(L_i)\in\End_{-2i}(V)$ for all $i\leq 1$, we may introduce constants $a_{\mu}^i$ by the defining relations
	$$
	\rho(L_i)(v_{\mu})\,=\, a_{\mu}^i v_{\mu+2i} \qquad \text{ for }\mu\in\xi \, , \, i\leq 1
	\, .
	$$
Then, the equation~(\ref{eq:rho(L_i)def2}) determines $a_{\mu}^i$ recursively and the explicit expressions for these coefficients are 
	{\small $$
	a_{\mu}^i \,=\,
	 \frac{(-1)^i}{2}
	\big( i(\pm \sqrt{\tau}-1)-\mu \big)
	\, P\big(\frac12(1+\mu\pm \sqrt{\tau}), i\big)
	$$}
for $\mu\in\xi$ and $i\leq 1$. It is worth noticing that both solutions coincide for $\tau=0,1$.

\emph{Step 4}. The above expressions for $a_{\mu}^i $  make sense as long as:
     $$
     \frac12(1+\mu\pm \sqrt{\tau}) + i 
     \,\neq \, 0\quad \forall \mu\in\xi \quad \forall i<0
     $$
Note that this condition is equivalent to $1\mp \sqrt{\tau}\notin \xi $; and thus, it depends whether the root $1\mp \sqrt{\tau}$ of $(x-1)^2-\tau$ lie in $\xi$.
	
\emph{Step 5}. It is now a straightforward computation that the identity of the second item of Theorem~\ref{thm:Witt<equivalent}  holds for all $i,j\geq 0$. 

Recalling Theorem~\ref{subset:decompCategory}, one obtains the statement for ${\mathbb V}(\xi,\tau)$ simple. 
\end{proof}

\begin{prop}\label{prop:denseWitt}
The dense $\sl(2)$-module ${\mathbb V}(\xi,\tau)$ does not admit a compatible structure of $\Witt$-module (resp.  $\Vir$-module) if and only if the roots of the polynomial $ (x-1)^2-\tau$ lie in $\xi$.

If the root $1\pm \sqrt{\tau}$ does not lie in $\xi$, then ${\mathbb V}(\xi,\tau)$ is a $\Witt$-module (resp.  $\Vir$-module) via the map:
	{\small $$
	\rho^{\pm}(L_i)v_{\mu} \,:=\,
 \Big(
	  \frac{(-1)^i}{2}
	\big( i(\pm \sqrt{\tau}-1)-\mu \big)
	\, P\big(\frac12(1+\mu\pm \sqrt{\tau}), i\big)
	\Big) v_{\mu+2i}
	$$}
where $i\in{\Z}$ and $\mu\in\xi$ (and $\rho^{\pm}(C)=0$ in the case of $\Vir$).

In particular, if ${\mathbb V}(\xi,\tau)$ is simple, then it carries exactly
 two structures for $\tau\neq 0,1$ and one for $\tau=0,1$. 
\end{prop}

\begin{proof}
Let us begin with the case of $\Witt$. The proof is based in the explicit expressions obtained in Propositions~\ref{prop:denseW>} and~\ref{prop:denseW<} and on the compatibility criterion given in Proposition~\ref{prop:LieAlgWitt}; namely, $[S,T]=2\sigma(h)$.

First, we note that if $\rho_{>}^{+}$ and $\rho_{<}^{+}$ are defined, then they agree on $\Witt_{>}\cap \Witt_{<}$ and $[\rho_{>}^{+}(L_2), \rho_{<}^{+}(L_{-2})]\,=\, 4 \rho_{<}^{+}(L_{0})$ .  Thus, according to Proposition~\ref{prop:LieAlgWitt},  $\rho_{>}^{+}$ and $\rho_{<}^{+}$ are glued and yield
	$$
	\rho^{+} \,\in\, \Hom_{\text{\bf LieAlg}_\bullet}(\Witt, \End(V))
	\, .
	$$
Similarly, $\rho_{>}^{-}$ and $\rho_{<}^{-}$ yield $\rho^{-} \in \Hom_{\text{\bf LieAlg}_\bullet}(\Witt, \End(V))$.

On the other hand,  although $\rho_{>}^{-}$ and $\rho_{<}^{+}$ agree on $\Witt_{>}\cap \Witt_{<}$ one checks that $ [\rho_{>}^{-}(L_2), \rho_{<}^{+}(L_{-2})]\,\neq\, 4 \rho_{<}^{+}(L_{0})$ for $\tau\neq 0,1$ and, thus, they do not yield a representation of $\Witt$. One proceeds analogously with 	$\rho_{>}^{+}$ and $\rho_{<}^{-}$.

Let us now check the statement for the $\Vir$-algebra. For this goal, one requires Propositions~\ref{prop:denseW>}, \ref{prop:denseW<} and Theorem~\ref{thm:LieAlgVirasoro2}. First, note that the previous representations of $\Witt$, $\rho^{\pm}$, can be trivially extended to $\Vir$ by setting $\rho^{\pm}(C)=0$. Second, we wonder whether in this case $\rho_{>}^{-}$ and $\rho_{<}^{+}$ could define a representation of $\Vir$. Having in mind Theorem~\ref{thm:LieAlgVirasoro2}, let $K:=4 \sigma(h)-2[\rho_{<}^{+}(L_{-2}), \rho_{>}^{-}(L_2)]$. Since $K$  does not commute with  $\rho_{<}^{+}(L_{-2})$, we conclude that they do not yield a representation of $\Vir$. The same holds for $\rho_{>}^{+}$ and $\rho_{<}^{-}$. 
\end{proof}

\begin{rem}
Analogously to Proposition~\ref{prop:DenseasWitt-Intermediate} and using \cite[Theorem~1]{Mathieu}, one can determine a $\Vir$-module of the intermediate series isomorphic to the module $({\mathbb V}(\xi,\tau),\rho^{\pm})$ . 
\end{rem}

Arguing analogously as in Theorem \ref{thm:categorydenseW>}, we obtain the following

\begin{thm}
Consider the category $\bar\Mf^{\xi,\tau}$  for $\tau\neq (\mu+1)^2$ for all $\mu\in\xi$ and $\tau\neq 0$. 

A root $ \sqrt{\tau}\in\C$ determines a fully faithful functor
	$$
	F_{\sqrt{\tau}}: \bar\Mf^{\xi,\tau} 
	\,\to \, 
	 {\Witt_{<}\text{--\,mod}}	
	$$
such that $G\circ F_{\sqrt{\tau}}=\operatorname{Id}_{\bar\Mf^{\xi,\tau} }$ where $G:{\Witt_{<}\text{--\,mod}}   \to  {\sl(2)\text{--\,mod}}$ is the forgetful functor. 

The same statement holds for ${\Witt\text{--\,mod}}$ and ${\Vir\text{--\,mod}}$ 
\end{thm}

\subsection{Second Case}\label{subsec:second}

Now, we focus on the category  $\bar\Mf^{\xi,\tau}$ for $\tau= (\lambda+1)^2$ for exactly one $\lambda\in\xi $. Recall that in this case $\lambda\notin\Z$ and that this category has two simple objects; namely, $M(\lambda)$ (the Verma module, see (\ref{eq:actionVerma})) and  $\bar M(\lambda+2)$ (see (\ref{eq:actionbarM})).

Although we are able to show that many objects of $\bar\Mf^{\xi,\tau}$ do admit module structures over Witt and Virasoro algebras, we exhibit cases of $\sl(2)$-modules that do not admit such extensions. The full characterization of such modules remains open and it deserves further research.

\subsubsection{$M(\lambda)$ and $\bar M(\lambda+2)$}

\begin{prop}\label{prop:VermaW>}
Let $\lambda\in\C$. The Verma $\sl(2)$-module $M(\lambda)$  admits exactly one compatible structure of $\Witt_>$-module. The structure is explicitly given by:
	{\small $$
	\rho_{>}^- (L_i)w_{j} \,:=\,  
	\Big(
	 -\frac{1}{2}
	\big(2(i-j) +(i+1)\lambda\big)
	\, P\big(j-i+1, i\big)\Big) w_{j-i}
	 \quad j-i\geq 0
		$$}
and, with respect to $\rho_{>}^-$ , $M(\lambda)$ is a highest weight $\Witt_>$-module and $-\frac12\lambda$ is the highest weight. 
\end{prop}

\begin{proof}
The proof relies essentially on the Steps of  the proof of Proposition~\ref{prop:denseW>}. Let us be more precise. Given $M(\lambda)$ as above, let us consider $\xi:=\bar \lambda\in\C/2\Z$, i.e. the class of $\lambda$, $\tau:=(\lambda +1)^2$ and:
	$$
	v_{\lambda-2j}\,:=\, w_j \qquad \text{ for } j=0,1,2,\dots
	$$
and observe that the action of $\sl(2)$ on the basis $\{v_{\lambda-2j}\vert  j=0,1,2,\dots\}$ coincides formally with relations~(\ref{eq:dense}). Steps 1 to 4 the proof of Proposition~\ref{prop:denseW>} are concerned with solving some difference equations which, under the above replacement, coincide with the equations of this case.

Summing up, if we take the general solution given in Step 3 of the proof of Proposition~\ref{prop:denseW>} and substitute $\pm\sqrt{\tau}$ by $\pm(\lambda +1)$ and $\mu$ by $\lambda-2j$, we get that there is at most two compatible structures; explicitly
	{\small $$
	\rho_{>}^+ (L_i)w_{j} \,:=\,  \Big(
	 \frac{(-1)^i}{2}
	\big(2j+(i-1)\lambda\big)
	\, P\big(1-j+\lambda, i\big)\Big) w_{j-i},
	\quad i\geq -1 \, , \, j-i\geq 0,
	$$}and:
	{\small $$
	\rho_{>}^- (L_i)w_{j} \,:=\,  \Big(
	 -\frac{(-1)^i}{2}
	\big(2(i-j) +(i+1)\lambda\big)
	\, P\big(-j, i\big)\Big) w_{j-i},
	\quad i\geq -1\, , \, j-i\geq 0.
	$$}

However, $\rho_{>}^+$ does defines such structure if and only if $\rho_{>}^+ (L_i)w_{j}=0$ for $j-i<0$, $j\geq 0$ and $i\geq -1$. And this is true only when $\lambda =-1$ and, accordingly, $\tau=(\lambda+1)^2=0$ and, in this case, $\rho_{>}^+ =\rho_{>}^-$ so it suffices to study  $\rho_{>}^-$ . On the other hand, since 
	$$
	 (-1)^i
	 P\big(-j, i\big)
	\,=\, 
	P\big(j-i+1, i\big)	
	$$
and $P\big(j-i+1, i\big)=0$ for $j-i<0$ and $i\geq -1$, the map $\rho_{>}^-$ always yields a $\Witt_{>}$-module structure on $M(\lambda)$.
\end{proof}

%
%
%

\begin{thm}\label{thm:barVermaW>}
The  module $\bar M(\lambda+2)$  admits at most two compatible structures of $\Witt_>$-modules; namely,
	{\small $$
	\bar \rho_{>}^+ (L_i)w_{j} \,:=\,  \Big(
	 \frac{-1}{2}
	\big((1-i)\lambda+2(j+1)\big)
	\, P\big(j+i+1, -i\big)\Big) w_{j+i},
	\quad  i\geq -1  , \lambda \in\C
	$$}and:
	{\small $$
	\bar \rho_{>}^- (L_i)w_{j} \,:=\,  \Big(
	 \frac{-1}{2}
	\big((1+i)\lambda + 2(j+1+i)\big)
	\, P\big(\lambda+2+j+i, -i\big)\Big) w_{j+i},
	\quad i\geq -1, 
	$$}with $\lambda\neq -3,-4,\cdots$.
Both extended structures coincide for $\lambda =-1,0$.

Furthermore,  $( M(\lambda+2) , \bar \rho_{>}^+)  $ and  $( M(\lambda+2) , \bar \rho_{>}^-)  $ ) are lowest weight $\Witt_>$-modules and $-\frac12(\lambda+2)$  is their lowest weight. 
\end{thm}

\begin{proof}
In order to reduce this case to the previous results, the Chevalley involution will be considered.

\emph{Step 1}. Note that the map $L_i\mapsto - L_{-i}$ is a Lie-algebra automorphism of $\Witt$, which induces the Chevalley involution on $\sl(2)$; that is,  $f\mapsto -e$, $h\mapsto -h$ and $e\mapsto -f$. Thus, for an $\sl(2)$-module $(V,\sigma)$, where $\sigma:\sl(2)\to \End(V)$,  define another $\sl(2)$-module $\widehat{V}:=(V,\hat \sigma)$ to be the vector space $V$ endowed with the  action given by the composition of the previous involution followed by $\sigma$. Further, sending $\rho$ to $\hat \rho$, we obtain a $1-1$-correspondence:
    {\small $$
    \left\{
    \begin{gathered}
    \rho\in \Hom_{\text{\bf LieAlg}_\bullet}(\Witt_>, \End(V))
    \\
    \text{such that } \rho\vert_{\sl(2)}=\sigma
    \end{gathered}\right\}
    \overset{1-1}\longrightarrow
    \left\{
    \begin{gathered}
    \rho\in \Hom_{\text{\bf LieAlg}_\bullet}(\Witt_<, \End(\widehat V))
    \\
    \text{such that } \rho\vert_{\sl(2)}=\hat\sigma
    \end{gathered}\right\}
    $$}

\emph{Step 2}. Letting $\sigma$ be the action given in equation~(\ref{eq:actionbarM}), the action on $\widehat{\bar M(\lambda+2)}$ can be described as follows:
	$$
	\begin{aligned}
    \hat\sigma(f)(w_j)\,&:=\, -\sigma(e)(w_j)\,=\,
        - w_{j+1}
    \\
    \hat\sigma(h)(w_j)\,&:=\, -\sigma(h)(w_j)\,=\,
        -(\lambda+2j+2) w_j
        \\
	\hat\sigma(e)(w_j)\,&:=\,- \sigma(f) (w_j)\,=\,
        j(\lambda+j+1)w_{j-1} \qquad j\geq 0
        \end{aligned}
        $$

Let us set $\xi:= - \bar \lambda\in\C/2\Z$, i.e. the class of $-\lambda$ and $\tau:=(\lambda +1)^2$. Consider the vector space generated by:
	$$
	V\,:=\, <\{ v_{-(\lambda+2)-2j}\, \text{ for } j=0,1,2,\dots\}>
	$$
endowed with the action given by equation~(\ref{eq:dense}). Observe that $\frac14(\tau-(\mu+1)^2) = -j(\lambda +j+1)$ for $\mu=-(\lambda+2)-2j$, and that the linear map:
	$$
    \begin{aligned}
    \widehat{\bar M(\lambda+2)} \,&\longrightarrow V
    \\
    w_j & \longmapsto  \, (-1)^j v_{-(\lambda+2)-2j} \qquad \text{ for } j=0,1,2,\dots
    \end{aligned}
	$$
establishes an isomorphism of $\sl(2)$-modules.

\emph{Step 3}. We may formally use the Steps in the proof of Proposition~\ref{prop:denseW<}, and we get:
	{\small $$
	\rho_{<}^{\pm}(L_i)v_{\mu} \,:=\,
 \Big(
	  \frac{(-1)^i}{2}
	\big( i(\pm \sqrt{\tau}-1)-\mu \big)
	\, P\big(\frac12(1+\mu\pm \sqrt{\tau}), i\big)
	\Big) v_{\mu+2i}
	$$}
where $i\leq 1$ and $\mu\in\xi$. It remains to check whether $P\big(\frac12(1+\mu\pm \sqrt{\tau},i)$ is well defined in our case. Substitute $\sqrt{\tau}$ by $(\lambda +1)$ and $\mu$ by $-(\lambda+2)-2j$, and get:
	{\small $$
	\rho_{<}^+ (L_i)w_{j} \,:=\,  \Big(
	 \frac{(-1)^i}{2}
	\big((1+i)\lambda+2(j+1)\big)
	\, P\big(-j, i\big)\Big) w_{j-i},
	\quad \text{ for } i\leq 1  , \lambda \in\C,
	$$}and:
	{\small $$
	\rho_{<}^- (L_i)w_{j} \,:=\,  \Big(
	  \frac{(-1)^i}{2}
	\big((1-i)\lambda + 2(j+1-i)\big)
	\, P\big(-\lambda-j-1, i\big)\Big) w_{j-i},
	\quad i\leq 1,
	$$}with $\lambda\neq -3,-4,\cdots.$
Note that both maps coincide for $\lambda =-1,0$.

\emph{Step 4}. Computing $\bar\rho_{>}^+:= \widehat{\rho_{<}^+}$ and $\bar\rho_{>}^-:= \widehat{\rho_{<}^-}$, the result follows.
\end{proof}

\subsubsection{A counterexample}\label{subset:counter}

The category $\bar\Mf^{\xi,\tau}$ hides some subtleties. 
In what follow, we will construct an object of $\bar\Mf^{\xi,\tau}$, of length $3$, which does not admit a compatible structure of $\Witt_{>}$-module. For this goal, recall that the objects of this category have been described in \cite[\S3.8]{Mazor}. For $\mu,\mu'\in\xi$ we write $\mu\leq \mu'$ if and only if $\mu'-\mu+2\in 2{\mathbb N}$. 

Let us consider
	{\small $$
	M:= \oplus_{\mu\in\xi} M_{\mu}
	\qquad\text{ where }\quad 
	M_{\mu} \,:=\, 
	\begin{cases}
	\C\oplus \C & \quad \text{ for }  \mu\leq \lambda 
	\\
	\C & \quad \text{ for }   \mu \geq \lambda + 2
	\end{cases}
	$$}
with  the action of $\sl(2)$, $\sigma$ (which will be omitted if no confusion arises), defined by the matrices
	$$
	f_{\mu}: M_{\mu}\to M_{\mu-2} 
	\qquad\text{ where } f_{\mu}\,:=\, 
	\begin{cases}
	\operatorname{Id} & \quad \text{ for }  \mu\neq \lambda +2 
	\\
	\mbox{$\begin{pmatrix} 0 \\ \frac12 \end{pmatrix}$} 	& \quad\text{ for }   \mu= \lambda +2
	\end{cases}
	$$
$h_{\mu}: M_{\mu}\to M_{\mu} $ is the homothety of	ratio $\mu$, and $e_{\mu}: M_{\mu}\to M_{\mu+2} $, where
	{\small $$ e_{\mu}\,:=\, 
	\begin{cases}
	 \frac14 ((\lambda+1)^2-(\mu+1)^2) \operatorname{Id} +\frac14\mbox{$\begin{pmatrix} 0 & 1 \\ 0 & 0 \end{pmatrix}$} & \quad \text{ for }  \mu\leq \lambda-2  
	\\
	\mbox{$\begin{pmatrix}  \frac12 & 0 \end{pmatrix}$} 	& \quad \text{ for }   \mu= \lambda 
	\\
	 \frac14 ((\lambda+1)^2-(\mu+1)^2) & \quad \text{ for }   \mu \geq  \lambda + 2 
	\end{cases}
	$$}

It is left to the reader the proof that $M$ is an $\sl(2)$-module that belongs to $\bar\Mf^{\xi,\tau}$ with $\tau=(\lambda+1)^2$ and $\xi:=\bar\lambda\in\C/2\Z$. We assume that $\tau=(\mu+1)^2$ for exactly one $\mu\in\xi$ and, thus, $\lambda\notin\Z$. Then, $M$  fits in the following exact sequence
	$$
	0\to M(\lambda) \to M \to {\mathbb V}(\bar \lambda,(\lambda+1)^2) \to 0
	\, .$$

\begin{prop}
The $\sl(2)$-module $M$ does not admit a compatible $\Witt_{>}$-module structure for $\lambda\in\C\setminus\Z$.
\end{prop}

\begin{proof}
Bearing in mind Proposition \ref{prop:Tiffrho} and Corollary \ref{thm:Witt>equivalent}, if suffices to show that there is no $T\in\End(M)$ such that $\ad(h)(T)=4T$, $\ad(f)(T)=3e$ and $\ad(T)(\ad(e)(T)) = -\frac16 \ad(e)^3(T)$.

Assume that $T$ exists. Let us denote by $T_{\mu}$ the matrix associated to the restriction $T\vert_{M_{\mu}}$ and observe that $T(M_{\mu})\subseteq M_{\mu+4}$. Let $A_{ij}$ the $(i,j)$-entry of the matrix $A$. 

Writing down the condition $[f,T]\vert_{M_{\mu}}=3e\vert_{M_{\mu}}$ for $\mu\leq \lambda-4 $ in terms of matrices, we obtain
	$$
	T_{\mu} - T_{\mu-2} \, = \,
	 \frac34  \begin{pmatrix} (\lambda+1)^2-(\mu+1)^2 & 1 \\ 0 &(\lambda+1)^2-(\mu+1)^2 \end{pmatrix}
	 $$
This difference equation determines $T_{\mu} $ for $\mu\leq \lambda-4 $ up to a $2\times 2$-matrix with  constant entries. This matrix can be computed from the condition $\ad(T)(\ad(e)(T)) = -\frac16 \ad(e)^3(T)$ for $\mu\leq \lambda-10$. At the end of the day, it turns out that there are two possibilities; on the one hand
	{
	$$
	T_{\mu}\,=\, 
	\begin{pmatrix}
	(T_{\mu})_{11}
	&
	\frac1{16(1+\lambda)} ( 19 + 3 \lambda - 12 \lambda^2 + 6 \mu + 6 \lambda  \mu)
	\\
	0 & 
	(T_{\mu})_{22}
	\end{pmatrix}
	$$}
with 
	{\small $$(T_{\mu})_{11}= (T_{\mu})_{22}= 
\frac1{16}( -15 + 38 \lambda + 3 \lambda ^2 - 8 \lambda ^3 - 16 \mu + 12 \lambda \mu + 6 \lambda ^2 \mu -      12 \mu ^2 - 2 \mu ^3 ) $$}
and, on the other hand
	{	$$
	T_{\mu}\,=\, 
	\begin{pmatrix}
	(T_{\mu})_{11}
	&
	\frac1{16(1+\lambda)} ( 35 + 51 \lambda + 12 \lambda^2 + 6 \mu + 6 \lambda  \mu)
	\\
	0 & 
	(T_{\mu})_{22} 
	\end{pmatrix}
	$$}
with
	{\small $$
	(T_{\mu})_{11}= (T_{\mu})_{22}= 
	 \frac1{16}( -15 + 70 \lambda + 51 \lambda ^2 + 8 \lambda ^3 - 16 \mu + 12 \lambda \mu + 6 \lambda ^2 \mu -      12 \mu ^2 - 2 \mu ^3 ) 
	$$}
where $\mu\leq \lambda-4$.

Similarly, the condition $[f,T]\vert_{M_{\lambda-2}}=3e\vert_{M_{\lambda -2 }}$ yields
	$$
	\begin{pmatrix} 0 \\ \frac12 \end{pmatrix} T_{\lambda-2} \,-\, 
	T_{\lambda-4} \operatorname{Id} \,=\, 
	3  \begin{pmatrix} \lambda & \frac14 \\ 0 & \lambda \end{pmatrix}
	$$
Let us deal with the first case, being the second one similar. Looking at the first row of the above matrices  bearing in mind that the first row of ${\small \begin{pmatrix} 0 \\ \frac12 \end{pmatrix} T_{\lambda-2} }$ is $0$,  one obtains a system of two equations
	{\small 
	$$ 
	\begin{aligned}
	\frac3{16}( -5 + 2 \lambda + 17 \lambda ^2 +4 \lambda ^3 ) \, 
	&=\, 3\lambda 
	\\
	\frac1{16(1+\lambda)} ( 11 + 33 \lambda +18 \lambda^2 )\, &=\, \frac34
	\end{aligned}
	$$}which has no solution for $\lambda$. Thus, $T$ can not exist independently of the choice of $\lambda$. 
\end{proof}

\subsubsection{Extending to $\Witt_<$, $\Witt$ and $\Vir$}

\begin{thm}\label{thm:VermaW<}
The Verma $\sl(2)$-module $M(\lambda)$  admits at most two compatible structures of $\Witt_<$-module. Moreover,  these structures are explicitly given by:
	{\small $$
	\rho_{<}^+ (L_i)w_{j} \,:=\,  \Big(
	 \frac{(-1)^i}{2}
	\big(2j+(i-1)\lambda\big)
	\, P\big(1-j+\lambda, i\big)\Big) w_{j-i}
	\qquad i\leq 1 \text{ for }\lambda\neq 1,2,\ldots
	$$}
and:
	{\small $$
	\rho_{<}^- (L_i)w_{j} \,:=\,  \Big(
	 -\frac{1}{2}
	\big(2(i-j) +(i+1)\lambda\big)
	\, P\big(j-i+1, i\big)\Big) w_{j-i}
	\qquad i\leq 1\text{ for }\lambda \in\C\, .
	$$}
Both coincide for $\lambda =-1, -2 $.

In particular, if $M(\lambda)$  is simple, it admits exactly two extended structures for $\lambda \neq -1, -2$ and one for $\lambda =-1,-2$.
\end{thm}

\begin{proof}
We proceed as in the proof of Proposition~\ref{prop:denseW<}; that is, it suffices to substitute $\pm\sqrt{\tau}$ by $\pm(\lambda+1)$ and $\mu$ by $\lambda-2j$ in that statement and recall, from Theorem~\ref{subset:decompCategory}, that $(x-1)^2-\tau$ has exactly one root in $\xi$. 

Note that in the case of dense $\sl(2)$-modules both structures coincide for $\tau=0,1$, which correspond to $\lambda =-2,-1$. 

Finally, since $M(\lambda)$  is simple if and only if $\lambda\neq 0,1,2,\ldots$, one concludes. 
\end{proof}

\begin{thm}
The  module $\bar M(\lambda+2)$  admits exactly one compatible structure of $\Witt_<$-modules; namely,
	{\small $$
	\bar \rho_{>}^- (L_i)w_{j} \,:=\,  \Big(
	 \frac{-1}{2}
	\big(2(1+j) +(1-i)\lambda\big)
	\, P\big(j+i+1, -i\big)\Big) w_{j+i}
	\qquad i\leq 1
	$$}
\end{thm}

\begin{proof}
Recalling Proposition~\ref{prop:VermaW>} and proceeding as in Theorem~\ref{thm:barVermaW>}, the result follows. 
\end{proof}

\begin{thm}
The Verma $\sl(2)$-module $M(\lambda)$  admits a unique compatible structure of $\Witt$-module; namely:
	{\small $$
	\rho_{<}^- (L_i)w_{j} \,:=\,  \Big(
	 -\frac{1}{2}
	\big(2(i-j) +(i+1)\lambda\big)
	\, P\big(j-i+1, i\big)\Big) w_{j-i}
	$$}

Moreover, $M(\lambda)$  admits a unique compatible structure of $\Vir$-module where  $\rho_{<}^- (L_i)$ is as above and $\rho_{<}^-(C)=0$. Moreover, with respect to this structure, $M(\lambda)$ is a highest weight $\Vir$-module of highest weight $(0, -\frac12\lambda) $. 
\end{thm}

\begin{proof}
It is analogous to the proof of Proposition~\ref{prop:denseWitt}.
\end{proof}

\begin{thm}
The lowest weight $\sl(2)$-module $\bar M(\lambda+2)$  admits a unique compatible structure of $\Witt$-module; namely:
	{\small $$
	\rho_{<}^- (L_i)w_{j} \,:=\,  \Big(
	 -\frac{1}{2}
	\big(2(1+j) +(1-i)\lambda\big)
	\, P\big(j+i+1, -i\big)\Big) w_{j+i}
	$$}

Moreover,  $\bar M(\lambda+2)$ admits a unique compatible structure of $\Vir$-module where  $\rho_{<}^- (L_i)$ is as above and $\rho_{<}^-(C)=0$. Moreover, with respect to this structure, $\bar M(\lambda+2)$ is a lowest weight $\Vir$-module of lowest weight $(0, -\frac12(\lambda+2)) $. 
\end{thm}

\begin{proof}
It is analogous to the proof of Proposition~\ref{prop:denseWitt}.
\end{proof}

\subsection{Third Case}

The last situation corresponds to $\tau= (\mu+1)^2=(\mu+2n+1)^2$ for some $n\in{\mathbb N}$. If this is the case, then $\mu=-n-1$, $\tau=n^2$ and   $\bar \Mf^{\xi,\tau}$ has three simple objects, namely $M(-n-1), {\mathbf V}^{(n)}$ and $\bar M(n+1)$. From Remark \ref{rem:RepreAreInfDim}, we know that ${\mathbf V}^{(n)}$  has no non-trivial representation of $\Witt_>$, $\Witt_>$, $\Witt$ or $\Vir$. The cases of $M(-n-1)$ and $\bar M(n+1)$ have been studied in \S\ref{subsec:second} since most statements of that section did not assumed simplicity. 

Regarding general $\sl(2)$-modules in this category, let us mention a couple of cases. For instance,  $M(n-1)$, which is an extension of $M(-n-1)$ by $ {\mathbf V}^{(n)} $, is a length $2$ $\sl(2)$-module admitting a compatible structure of $\Witt_>$-module and ${\mathbb V}(\xi,\tau)$ is a length $3$ $\sl(2)$-module which also admits such an extension. On the other hand, following the ideas of \S\ref{subset:counter}, one could produce examples of modules in $\bar \Mf^{\xi,\tau}$ which do not admit extensions.



\end{document}